%% file: main.tex
\documentclass{article}
\usepackage{amsfonts}
\usepackage{mathrsfs}
\usepackage{amsthm}
\usepackage[a4paper, total={6in, 8in}]{geometry}
\usepackage{graphicx}
\usepackage{color}
\usepackage{enumerate}
\graphicspath{{figure/}}
\usepackage{indentfirst,latexsym,bm}
\usepackage[ruled,linesnumbered]{algorithm2e}
\usepackage[colorinlistoftodos]{todonotes}
\usepackage{amsmath}
\usepackage{extarrows}
\usepackage{enumitem}
\usepackage{thmtools}
\usepackage{nicematrix,tikz}
\usepackage{hyperref}
\usepackage{soul} 
\usepackage[capitalise]{cleveref}
\crefname{equation}{}{}

\allowdisplaybreaks[4]
\numberwithin{equation}{section}
\usepackage{tikz}
\usepackage{scalerel} 
\crefname{enumi}{}{}   \Crefname{enumi}{}{}
\crefname{enumii}{}{}  \Crefname{enumii}{}{}
\crefname{enumiii}{}{} \Crefname{enumiii}{}{}
\crefname{enumiv}{}{}  \Crefname{enumiv}{}{}
\crefformat{enumi}{#2#1#3}
\crefformat{enumii}{#2#1#3}
\crefformat{enumiii}{#2#1#3}
\crefformat{enumiv}{#2#1#3}

\newlength{\bracewidth}

\theoremstyle{plain}

\newtheorem{thm}{Theorem}[section]
\declaretheorem[name=Proposition, sibling=thm, refname={Proposition,Propositions}]{prop}
\declaretheorem[name=Corollary,   sibling=thm, refname={Corollary,Corollaries}]{coro}
\declaretheorem[name=Lemma,       sibling=thm, refname={Lemma,Lemmas}]{lemma}

\declaretheorem[name=Fact,        sibling=thm, refname={Fact,Facts}]{fact}

\theoremstyle{plain}
\declaretheorem[
  name=Definition,
  sibling=thm,
  refname={Definition,Definitions},
  style=definition  
]{defn}
\declaretheorem[
  name=Remark,
  sibling=thm,
  refname={Remark,Remarks},
  style=remark
]{remark}

\newcommand{\R}{\mathbb{R}}

\newtheorem{exmp}[thm]{Example}
\newcommand{\tr}{\mathrm{tr}}

\newcommand{\diag}{\mathrm{diag}}
\newcommand{\Diag}{\mathrm{Diag}}

\newcommand{\rank}{\mathrm{rank}}
\newcommand{\revise}[1]{\textcolor{black}{#1}}

\newcommand{\bS}{\mathbb{S}}
\def\sp{\mathcal{P}}
\def\gp{\mathfrak{P}}

\newcommand{\cS}{\mathcal{S}}
\newcommand{\cO}{\mathcal{O}}

\newcommand{\fS}{\mathfrak{S}}

\newcommand{\cJ}{\mathcal{J}}

\newcommand{\RIP}{\mathrm{RIP}}
\newcommand{\cR}{\mathcal{R}}

\newcommand{\n}{\mathbb{N}}
\newcommand{\wdiag}{\widetilde{\rm diag}}
\newcommand{\wDiag}{\widetilde{\rm Diag}}

\newcommand{\ie}{{\it i.e.}}
\newcommand{\eg}{{\it e.g.}}

\usepackage{subfigure}
\usepackage{amssymb}

\begin{document}
\title{Burer-Monteiro factorizability of nuclear norm regularized optimization}
\author{
Wenqing Ouyang\thanks{Columbia University, New York, USA ({\tt wo2205@columbia.edu}).}
\and
Ting Kei Pong\thanks{The Hong Kong Polytechnic University, Hong Kong, People's Republic of China ({\tt tk.pong@polyu.edu.hk}). The research of this author was partially supported by the Hong Kong Research Grants Council PolyU 15300122 and the PolyU internal grant 4-ZZPJ.}
\and
Man-Chung Yue\thanks{The Hong Kong Polytechnic University, Hong Kong, People's Republic of China ({\tt manchung.yue@polyu.edu.hk}). The research of this author was partially supported by the Hong Kong Research Grants Council under the GRF project 17309423.}
}
\date{\today} 
\maketitle

\begin{abstract}
  This paper studies the relationship between the nuclear norm-regularized minimization problem, which minimizes the sum of a $C^2$ function $h$ and a positive multiple of the nuclear norm, denoted by $f$, and its factorized problem obtained by the Burer-Monteiro technique. We are interested in deriving conditions that ensure every second-order stationary point of the factorized problem corresponds to a global minimizer of $f$, a property we call the $r$-factorizability of $f$ in this paper. Under suitable restricted isometry property (RIP) type assumptions on $h$, we prove the $r$-factorizability of $f$. Moreover, the RIP constant in our paper is tight, in the sense that \revise{concrete non-$r$-factorizable $f$ can be constructed when the RIP-type assumption fails to hold}. Our technique for constructing such examples is novel and may be of independent interest: specifically, we use a variant of the Von Neumann's trace inequality and relate the existence of such examples to the optimal value of a quadratic program involving the RIP constant, then we explicitly solve this optimization problem to \revise{identify the parameter regimes in which such worst-case counterexamples can be constructed}.

\end{abstract}
\input{sec1}
\input{sec3}
\appendix

\section{Proof of \cref{thm5-9}}
\label{appa}
In this section, we finish the key step of the proof of \cref{thm5-9} by solving \cref{eq:opt_sc}. We observe that the objective can be rewritten as a sum with the $i$th summand depending only on $(x_i,g_i,y_{\tau(i)},v_{\tau(i)})$.\footnote{Specifically, notice that the fourth sum in the objective can be rewritten as $\sum_{i=1}^m(Ly_{\tau(i)}+v_{\tau(i)})(\mu y_{\tau(i)}+v_{\tau(i)})$.} In addition, from the structure of the constraints in \eqref{eq:opt_sc}, we see that the terms $\{(x_i,g_i,y_{\tau(i)},v_{\tau(i)})\}_{i\in [m]}$ can be divided into 4 groups depending on whether $i\in [r]$ and $\tau(i)\in [r^*]$. These motivate the definitions of the next four associated index sets, for any fixed $\tau\in \gp_m$:
\begin{equation}
    \label{tau_indexsets}
    \cJ_{1}^\tau:=[r]\cap \tau^{-1}[r^*],~\cJ_{2}^\tau:=[r]\setminus \cJ_{1}^\tau,~\cJ_{3}^\tau:=([m]\setminus[r])\cap \tau^{-1}[r^*],~\cJ_{4}^\tau:=([m]\setminus[r])\setminus \cJ_{3}^\tau.
\end{equation}
Let $d_\tau:=|\cJ_2^\tau|$. Then, by the definition of $\{\cJ_i^\tau\}_{i\in [4]}$ in \cref{tau_indexsets}, we have 
\begin{equation}
    \label{bound_indexset}
  |\revise{\cJ_1^\tau}|=r-|\cJ_2^\tau|=r-d_\tau,~|\cJ_3^\tau|=r^*-|\cJ_1^\tau|=r^*-r+d_\tau,~|\cJ_4^\tau|=m-r-|\cJ_3^\tau|=m-r^*-d_\tau.
\end{equation}
To solve \eqref{eq:opt_sc}, our strategy is to introduce an auxiliary variable $w\in \R$ to transform \cref{eq:opt_sc} to the next equivalent form:
\begin{equation}
    \label{eq:opt_sc_introducew}
    \begin{aligned}
            &&\sup_{\begin{subarray}{c}
            x,g,y,v\in \R^{m}\\
            \tau\in \gp_m,w\in \R
            \end{subarray}} &\sum_{i=1}^m (L x_i+g_i)(\mu y_{\tau(i)}+v_{\tau(i)})+\sum_{i=1}^m (\mu x_i+g_i)(Ly_{\tau(i)}+v_{\tau(i)}) \\
            &&&-\sum_{i=1}^m(Lx_i+g_i)(\mu x_i+g_i)-\sum_{i=1}^m(Ly_i+v_i)(\mu y_i+v_i) \\
            &&s.t.\quad ~&\forall i\in [r],~x_i\geq w>0,~g_i=\lambda,\quad \forall i\in [r^*],~y_i>0,~v_i=\lambda,\\
            &&&\forall i\in [m]\setminus [r],~x_i=0,~g_i\in [0,\lambda+Lw],\\\
            &&&\forall i\in [m]\setminus [r^*], ~y_i=0,~v_i\in [0,\lambda]. 
    \end{aligned}
\end{equation}
Problem \eqref{eq:opt_sc} is equivalent to \eqref{eq:opt_sc_introducew} in the following sense: for any feasible solution $(x,g,y,v,\tau)$ of \eqref{eq:opt_sc}, $(x,g,y,v,\tau,\min_{i\in [r]}x_i)$ is a feasible solution of \revise{\eqref{eq:opt_sc_introducew}} having the same objective function value; for any feasible solution $(x,g,y,v,\tau,w)$ of \eqref{eq:opt_sc_introducew}, $(x,g,y,v,\tau)$ is a feasible solution of \eqref{eq:opt_sc} having the same objective function value. Consequently, we have the following result.
\begin{prop}
\label{propb2}
    There exists a feasible solution $(x,g,y,v,\tau)$ with $\|g\|_\infty>\lambda$ to \eqref{eq:opt_sc} having nonnegative objective function value if and only if there exists a feasible solution $(x,g,y,v,\tau,w)$ with $\|g\|_\infty>\lambda$ to \eqref{eq:opt_sc_introducew} having nonnegative objective function value.
\end{prop}
Next, we plan to fix $\tau$ and $w$ to analyze the optimal value and the optimal solution of \eqref{eq:opt_sc_introducew}. The reason to do so is that the optimization problem can be made separable when $\tau$ and $w$ are fixed. To simplify the calculation, we only consider the case where $\lambda=1$ in the next lemma.
\begin{lemma}
\label{lema1}
 Let $r^*\in [m]\cup\{0\}$, $r\in [m]$, and $\infty>L\geq \mu>0$. Let $\tau\in \gp_m$ and $w>0$. Let $\{\cJ_i^\tau\}_{i\in[4]}$ be defined in \cref{tau_indexsets}. Consider the following optimization problem:
      \begin{equation}
            \label{eq:opt_scfixedw}
            \begin{aligned}
                &&\sup_{x,g,y,v\in \R^{m}} & \sum_{i=1}^m (L x_i+g_i)(\mu y_{\tau(i)}+v_{\tau(i)})+\sum_{i=1}^m (\mu x_i+g_i)(Ly_{\tau(i)}+v_{\tau(i)}) \\
                &&&-\sum_{i=1}^m(Lx_i+g_i)(\mu x_i+g_i)-\sum_{i=1}^m(Ly_i+v_i)(\mu y_i+v_i)  \\
                &&s.t.\quad~  &\forall i\in [r],~x_i\geq w,~g_i=1,\quad \forall i\in [r^*],~y_i>0,~v_i=1,\\
                &&&\forall i\in [m]\setminus [r],~x_i=0,~g_i\in [0,1+Lw],\quad \forall i\in [m]\setminus [r^*], ~y_i=0,~v_i\in [0,1]. 
            \end{aligned}
        \end{equation}
        Then, the optimization problem \eqref{eq:opt_scfixedw} has optimal solutions, and the optimal value is         
        \begin{equation}\label{opt_val_val}
        \left(-L\mu|\cJ_2^\tau|+|\cJ_3^\tau|\frac{L(L-\mu)^2}{4\mu}\right)w^2.
        \end{equation}
        Moreover, the following statements are equivalent:
        \begin{itemize}
            \item $|{\cal J}_3^\tau|>0$.
            \item For all the optimal solutions $(\bar x,\bar g,\bar y,\bar v)$ of \eqref{eq:opt_scfixedw}, we have $\|\bar g\|_\infty>1$.
            \item There exists one optimal solution $(\bar x,\bar g,\bar y,\bar v)$ of \eqref{eq:opt_scfixedw} such that $\|\bar g\|_\infty>1$.
        \end{itemize}
\end{lemma}
\begin{proof}
Using the definition of permutation, we notice that       
\begin{equation}
            \label{eq:permu_eq}
            \begin{aligned}
                  \sum_{i=1}^m (L y_i+v_i)(\mu y_i+v_i)=\sum_{i=1}^m(Ly_{\tau(i)}+v_{\tau(i)})(\mu y_{\tau(i)}+v_{\tau(i)}).
            \end{aligned}
\end{equation}
Substituting \cref{eq:permu_eq} into \cref{eq:opt_scfixedw}, we get that 
\begin{equation}
    \label{eq:opt_scfixedwpermu}
     \begin{aligned}
                &&\sup_{x,g,y,v\in \R^{m}} & \sum_{i=1}^m (L x_i+g_i)(\mu y_{\tau(i)}+v_{\tau(i)})+\sum_{i=1}^m (\mu x_i+g_i)(Ly_{\tau(i)}+v_{\tau(i)}) \\
                &&&-\sum_{i=1}^m(Lx_i+g_i)(\mu x_i+g_i)-\sum_{i=1}^m(Ly_{\tau(i)}+v_{\tau(i)})(\mu y_{\tau(i)}+v_{\tau(i)})  \\
                &&s.t.~~~~  &\forall i\in [r],~x_i\geq w,~g_i=1,\quad \forall i\in [r^*],~y_i>0,~v_i=1,\\
                &&&\forall i\in [m]\setminus [r],~x_i=0,~g_i\in [0,1+Lw],\quad \forall i\in [m]\setminus [r^*], ~y_i=0,~v_i\in [0,1]. 
    \end{aligned}
\end{equation}
Observe that \eqref{eq:opt_scfixedwpermu} can be decomposed into the next $m$ subproblems for each $i\in [m]$:
\begin{equation}
    \label{scsubsub}
    \begin{aligned}
           \sup_{x_i,g_i,y_i,v_i\in \R}~& (L x_i+g_i)(\mu y_{\tau(i)}+v_{\tau(i)})+  (\mu x_i+g_i)(Ly_{\tau(i)}+v_{\tau(i)}) \\
                &- (Lx_i+g_i)(\mu x_i+g_i)- (Ly_{\tau(i)}+v_{\tau(i)})(\mu y_{\tau(i)}+v_{\tau(i)}) \\
          s.t. ~~~~~~&~\begin{cases}
            x_i\geq w,~g_i=1,~y_{\tau(i)}>0,~v_{\tau(i)}=1,  & ~~  \text{if }i\in \cJ^\tau_1,   \\ 
            x_i\geq w,~g_i=1,~y_{\tau(i)}=0,~v_{\tau(i)}\in [0,1],  & ~~ \text{if }i\in \cJ^\tau_2,   \\
            x_i=0,~g_i\in [0,1+Lw],~y_{\tau(i)}>0,~v_{\tau(i)}=1,  & ~~ \text{if }i\in \cJ^\tau_3,  \\
            x_i=0,~g_i\in [0,1+Lw],~y_{\tau(i)}=0,~v_{\tau(i)}\in [0,1],  & ~~ \text{if }i\in \cJ^\tau_4,
          \end{cases}      
    \end{aligned}
\end{equation}
where we recall the definition of $\{\cJ^\tau_i\}_{i\in [4]}$ in \cref{tau_indexsets}. We now consider the solution and optimal value of each subproblem \cref{scsubsub} for fixed $i$:
        \begin{enumerate}
        \item $i\in \cJ^\tau_1$. Then \cref{scsubsub} takes the following form:
        \begin{equation}
            \label{eq:opt_scsub1}
            \begin{aligned}
                &&\sup_{x_i,y_{\tau(i)}} &(Lx_i+1)(\mu y_{\tau(i)}+1)+(\mu x_i+1)(L y_{\tau(i)}+1)\\
                &&&-(Lx_i+1)(\mu x_i+1)-(Ly_{\tau(i)}+1)(\mu y_{\tau(i)}+1) \\
                &&s.t.~&~x_i\geq w,~y_{\tau(i)}>0,
            \end{aligned}
        \end{equation}
        where we used the fact that $g_i$ and $v_{\tau(i)}$ are $1$. Denote the objective function of \cref{eq:opt_scsub1} by $S_1$. By direct calculation, we can rewrite $S_1(x_i,y_{\tau(i)})$ as:
        \[   S_1(x_i,y_{\tau(i)})=-L\mu(x_i-y_{\tau(i)})^2.                 \]
        Clearly, the optimal value of \eqref{eq:opt_scsub1} is $0$, and it is achieved if and only if
        \begin{equation}
        \label{solsubcase1}
            x_i=y_{\tau(i)}\geq w. 
        \end{equation}
        \item $i\in \cJ^\tau_2$. Then \cref{scsubsub} takes the following form:
        \begin{equation}
            \label{eq:opt_scsub2}
            \begin{aligned}
                &&\sup_{ x_i , v_{\tau(i)}  } & (Lx_i+1)v_{\tau(i)} + (\mu x_i+1)v_{\tau(i)}  - (Lx_i+1)(\mu x_i+1)- v_{\tau(i)}^2 \\
                && s.t.~~&x_i\geq w,~v_{\tau(i)}\in [0,1],
            \end{aligned}
        \end{equation}
        where we used the fact that $g_i$ and $y_{\tau(i)}$ are $1$ and $0$, respectively. Denote the objective function of \cref{eq:opt_scsub2} by $S_2$. By direct calculation, we can rewrite $S_2(x_i,v_{\tau(i)})$ as:
        \[   S_2(x_i,v_{\tau(i)})=-L\mu x_i^2+ (L+\mu)x_i(v_{\tau(i)}-1)- (v_{\tau(i)}-1)^2.                 \]
        First, we notice that $S_2$ is strictly decreasing on $[0,\infty)$ as a function of $x_i$ when $v_{\tau(i)}$ is fixed to be any value in $[0,1]$. This means that 
        \begin{equation}
            \label{s2minx}
            \sup_{x_i\geq w} S_2(x_i,v_{\tau(i)})=-L\mu w^2+ (L+\mu)w(v_{\tau(i)}-1)-(v_{\tau(i)}-1)^2, 
        \end{equation}
        where the optimal value is achieved if and only if $x_i=w$. Let $\tilde S$ denote the function on the right hand side of \cref{s2minx}. Then we see $\tilde S$ is strictly increasing as a function of $v_{\tau(i)}$ on $(-\infty,1]$ by using the elementary properties of quadratic functions. 
        Therefore, the optimal value of \eqref{eq:opt_scsub2} is $-L\mu w^2$, and it is achieved if and only if 
        \begin{equation}
            \label{solsubcase2}
               x_i=w,~v_{\tau(i)}=1. 
        \end{equation}
        \item $i\in \cJ^\tau_3$. Then \cref{scsubsub} has the following form:
        \begin{equation}
            \label{eq:opt_scsub3pre}
            \begin{aligned}
                &\sup_{g_i,y_{\tau(i)}} &&g_i(\mu y_{\tau(i)}+1)+g_i(Ly_{\tau(i)}+1)  -g_i^2-(Ly_{\tau(i)}+1)(\mu y_{\tau(i)}+1) \\
                &~~s.t.~&&~g_i\in [0,1+Lw],~y_{\tau(i)}>0,
            \end{aligned}
        \end{equation}
        where we used the fact that $x_i$ and $v_{\tau(i)}$ are $0$ and $1$, respectively. Denote the objective function of \cref{eq:opt_scsub3pre} by $S_3$. By direct calculation we can rewrite $S_3$ as follows 
        \begin{align}\label{S3defhahaha}
            S_3(g_i,y_{\tau(i)})&=\frac{(L-\mu)^2}{4L\mu}(g_i-1)^2-L\mu\left(y_{\tau(i)}-\frac{(L+\mu)(g_i-1)}{2L\mu}\right)^2\nonumber\\
            &= \frac{L(L-\mu)^2w^2}{4\mu}+\frac{(L-\mu)^2}{4L\mu}(g_i-(1+Lw))(Lw+g_i-1)\nonumber\\
            &~~~-L\mu\left(y_{\tau(i)}-\frac{(L+\mu)(g_i-1)}{2L\mu}\right)^2 .
        \end{align}
        Notice that $g_i-(1+Lw)\leq 0$ and $Lw+g_i-1>0$ when $g_i\in (1,1+Lw]$. We can thus see from the second expression in the above display that the optimal value of $S_3$ when $g_i\in (1,1+Lw]$ is $\frac{L(L-\mu)^2w^2}{4\mu}$; moreover, when $L>\mu$, the optimal value is achieved if and only if 
        \begin{equation}
            \label{solsubcase31}
            g_i=1+Lw,~y_{\tau(i)}=\frac{(L+\mu)w}{2\mu},
        \end{equation}
        while when $L=\mu$, the optimal value is achieved if and only if
        \begin{equation}
            \label{solsubcase32}
              g_i\in (1,1+Lw],~y_{\tau(i)}=\frac{(L+\mu)(g_i-1)}{2L\mu}. 
        \end{equation}
        On the other hand, when $g_i\leq 1$, notice that $y_{\tau(i)}>0$, and hence $y_{\tau(i)}-\frac{(L+\mu)(g_i-1)}{2L\mu}>|\frac{(L+\mu)(g_i-1)}{2L\mu}|$. Then we have from the first expression of $S_3$ in \eqref{S3defhahaha} that
         \begin{align*}
            S_3(g_i,y_{\tau(i)})<\frac{(L-\mu)^2}{4L\mu}(g_i-1)^2-L\mu\left(\frac{(L+\mu)(g_i-1)}{2L\mu}\right)^2=-(g_i-1)^2\leq 0.
        \end{align*}
        Consequently, the optimal value of \cref{eq:opt_scsub3pre} is $\frac{L(L-\mu)^2w^2}{4\mu}$, and is achieved as described in \eqref{solsubcase31} and \eqref{solsubcase32}.
        
        \item $i\in \cJ^\tau_4$. Then \cref{scsubsub} has the following form:
        \begin{equation}
            \label{eq:opt_scsub4}
            \begin{aligned}
                &\sup_{g_i,v_{\tau(i)}} &&2g_iv_{\tau(i)}-g_i^2-v_{\tau(i)}^2 \\
                &~~s.t.&&g_i\in [0,1+Lw],~v_{\tau(i)}\in [0,1],
            \end{aligned}
        \end{equation}
        where we used the fact that $x_i$ and $y_{\tau(i)}$ are $0$. Notice that the objective of the above problem is $-(g_i-v_{\tau(i)})^2$.
        Clearly, the optimal value of \eqref{eq:opt_scsub4} is $0$, and is achieved if and only if
        \begin{equation}
            \label{solsubcase4}
             g_i=v_{\tau(i)}\in [0,1]. 
        \end{equation}
    \end{enumerate}
    Consequently, by the solution sets given in \eqref{solsubcase1}, \eqref{solsubcase2}, \eqref{solsubcase31}, \eqref{solsubcase32} and \eqref{solsubcase4}, we know the solution set of \eqref{eq:opt_scfixedw} is nonempty. The optimal value is obtained by summing all the optimal values given in the four cases. Moreover, every solution $(\bar x,\bar g,\bar y,\bar v)$ of \eqref{eq:opt_scfixedw} satisfies $\|\bar g\|_\infty>1$ if and only if $|\cJ_3^\tau|>0$, according to the structure of $\bar g$ given in \eqref{solsubcase31}, \eqref{solsubcase32} and \eqref{solsubcase4}.
\end{proof}

\begin{prop}
    \label{lemma5-7}
    Let $r^*\in [m]\cup\{0\},r\in [m]$, and $\infty>L\geq \mu>0$. Let $G$ be the objective function of \cref{eq:opt_sc_introducew} and let $\kappa:=\frac{L}{\mu}\geq 1$. If $r^*$, $r$ and $\kappa$ satisfy any of the following conditions, then there is no feasible $(\bar x,\bar g,\bar y,\bar v,\bar\tau,\bar w)$ to \eqref{eq:opt_sc_introducew} satisfying $\|\bar g\|_\infty>\lambda$ and $G(\bar x,\bar g,\bar y,\bar v,\bar\tau,\bar w)\geq 0$.  
    \begin{itemize}
            \item[(1)] $r=m$.
            \item[(2)] $r\geq r^*$ and $ \min\{r,m-r^*\}>\frac{(\kappa-1)^2}{4}\min\{r^*,m-r\} $.
    \end{itemize}
    Otherwise, such a feasible solution exists.
\end{prop}
\begin{proof}
    By the change of variables $(x,g,y,v,\tau,w)\gets (x/\lambda,g/\lambda,y/\lambda,v/\lambda,\tau,w/\lambda)$, we see that \cref{eq:opt_sc_introducew} can be reduced to the following problem:
    \begin{equation}
        \label{eq:opt_sc1}
        \begin{aligned}
            &&\sup_{\begin{subarray}{c}
                x,g,y,v\in \R^{m}\\
                \tau\in \gp_m,w\in \R
            \end{subarray}} &\lambda^2\bigg[\sum_{i=1}^m (L x_i+g_i)(\mu y_{\tau(i)}+v_{\tau(i)})+\sum_{i=1}^m (\mu x_i+g_i)(Ly_{\tau(i)}+v_{\tau(i)}) \\
            &&&~~~~-\sum_{i=1}^m(Lx_i+g_i)(\mu x_i+g_i)-\sum_{i=1}^m(Ly_i+v_i)(\mu y_i+v_i)\bigg] \\
            &&s.t.~~&\forall i\in [r],~x_i\geq w>0,~g_i=1,\quad \forall i\in [r^*],~y_i>0,~v_i=1,\\
            &&&\forall i\in [m]\setminus [r],~x_i=0,~g_i\in [0,1+Lw],\\
            &&&\forall i\in [m]\setminus [r^*], ~y_i=0,~v_i\in [0,1]. 
        \end{aligned}
    \end{equation}
    Since dropping the constant $\lambda^2$ won't affect the sign of the function value and our claim only concerns the feasible set of \eqref{eq:opt_sc_introducew} and the {\em sign} of its objective value, we shall consider the following optimization problem instead:
    \begin{equation}
        \label{eq:opt_scfixtau1}
        \begin{aligned}
            &&\sup_{\begin{subarray}{c}
                x,g,y,v\in \R^{m}\\
                \tau\in \gp_m,w\in \R
            \end{subarray}} &\sum_{i=1}^m (L x_i+g_i)(\mu y_{\tau(i)}+v_{\tau(i)})+\sum_{i=1}^m (\mu x_i+g_i)(Ly_{\tau(i)}+v_{\tau(i)}) \\
            &&&-\sum_{i=1}^m(Lx_i+g_i)(\mu x_i+g_i)-\sum_{i=1}^m(Ly_i+v_i)(\mu y_i+v_i) \\
            &&s.t.~~ &\forall i\in [r],~x_i\geq w>0,~g_i=1,\quad \forall i\in [r^*],~y_i>0,~v_i=1,\\
            &&&\forall i\in [m]\setminus [r],~x_i=0,~g_i\in [0,1+Lw],\\
            &&&\forall i\in [m]\setminus [r^*], ~y_i=0,~v_i\in [0,1].  
        \end{aligned}
    \end{equation}
    Notice that when $\tau$ and $w$ are fixed, \cref{eq:opt_scfixtau1} becomes \cref{eq:opt_scfixedw}. Applying \cref{lema1}, setting $d_\tau=|\cJ^\tau_2|$ and recalling the definition of $\cJ^\tau_3$ in \eqref{bound_indexset}, we see that the solution set $\Omega_{w,\tau}$ of \eqref{eq:opt_scfixedw} is nonempty, and for all $(\bar x,\bar g,\bar y,\bar v)\in \Omega_{w,\tau}$, it holds that $\|\bar g\|_\infty>1$ if and only if $r^*-r+d_\tau>0$. 
    
Next, define the following function $H:\n_0\times \R_+ \to \R$:
        \begin{equation}
            \label{Htauw}
          H(d,w):=\left(-L\mu d+(r^*-r+d)\frac{L(L-\mu)^2}{4\mu}\right)w^2.
        \end{equation}
    Then in view of \eqref{bound_indexset} and \eqref{opt_val_val}, the optimal value of \eqref{eq:opt_scfixedw} is given by $H(|\cJ^\tau_2|,w)$.    
    Moreover, we see that \eqref{eq:opt_scfixtau1} is equivalent to the following problem 
    \begin{equation}
        \label{prob_reduced_sc}
        \sup_{w\in \R,~d\in \n_0}   H(d,w)\quad s.t. ~w>0,~  r-r^*\leq d\leq \min\{r,m-r^*\}, 
    \end{equation}
    where the bound for $d$ comes from the requirement that $|\cJ_i^\tau|\geq 0$ for $i\in [4]$ in \cref{bound_indexset}.
    
    We consider the following scenarios:
    \begin{itemize}
        \item[(S1)] The optimal value of \eqref{prob_reduced_sc} is nonpositive, and \eqref{prob_reduced_sc} has no feasible solution $(d,w)$ satisfying $H(d,w)\geq 0$ and $r^*-r+d>0$.
        
        In this scenario, we claim that \eqref{eq:opt_sc_introducew} has no feasible solution $(\bar x,\bar g,\bar y,\bar v,\bar \tau,\bar w)$ with $\|\bar g\|_\infty>\lambda$ and $G(\bar x,\bar g,\bar y,\bar v,\bar \tau,\bar w)\geq 0$, where $G$ is the objective of \cref{eq:opt_sc_introducew}. A short proof is provided below.
        
        Suppose such a feasible solution $(\bar x,\bar g,\bar y, \bar v, \bar\tau,\bar w)$ of \eqref{eq:opt_sc_introducew} exists, then either $(\bar x,\bar g,\bar y,\bar v)/\lambda$ is optimal for \eqref{eq:opt_scfixedw} with $w=\bar w/\lambda$ and $\tau=\bar\tau$, or $(\bar x,\bar g,\bar y,\bar v)/\lambda$ is not optimal. In the latter case, the optimal value of \eqref{prob_reduced_sc} must be positive. In the former case, we see from \cref{lema1} that $|\cJ_3^{\bar\tau}|>0$, and hence \eqref{prob_reduced_sc} has a feasible solution $(\tilde d,\tilde w)=(|\cJ^{\bar\tau}_2|,\bar w/\lambda)$ with $H(\tilde d,\tilde w)\geq 0$ and $r^*-r+\tilde d = r^* - r + |\cJ^{\bar\tau}_2| = |\cJ^{\bar\tau}_3|>0$ (see \eqref{bound_indexset}).  Both cases yield a contradiction.
        \item[(S2)] There exists a feasible solution $(d,\tilde w)$ of \cref{prob_reduced_sc} satisfying $H(d,\tilde w)\geq 0$ and $r^*-r+d>0$.
        
        In this scenario, \eqref{eq:opt_sc_introducew} has a feasible solution $(\bar x,\bar g,\bar y,\bar v,\bar\tau,\bar w)$ with $G(\bar x,\bar g,\bar y,\bar v,\bar\tau,\bar w)\geq 0$ and $\|\bar g\|_\infty>\lambda$. Indeed, we just need to take $\bar\tau\in \gp_m$ satisfying $|\cJ_2^{\bar\tau}|=d$, and then take $(\tilde x, \tilde g,\tilde y,\tilde v)$ to be the optimal solution of \eqref{eq:opt_scfixedw} with $\tau=\bar\tau$ and $w=\tilde w$, and set $(\bar x,\bar g,\bar y,\bar v,\bar w)=\lambda(\tilde x, \tilde g,\tilde y,\tilde v,\tilde w)$.
    \end{itemize}
    We note that the classification in (S1) and (S2) is \textit{not} complete, since we cannot say anything if the optimal value of \eqref{prob_reduced_sc} is positive and there is no feasible solution $(d,w)$ of \eqref{prob_reduced_sc} satisfying $H(d,w)\geq 0$ and $r^*-r+d>0$. Nevertheless, the two scenarios in (S1) and (S2) are enough for our proof.
    Consider the following cases on $r$, $r^*$ and $\kappa := L/\mu$.
    \begin{itemize}
        \item[Case 1:] $r=m$ or $r^*=0$. If $r=m$, then 
        every feasible point $(d,w)$ to \eqref{prob_reduced_sc} must satisfy $d = r-r^*$ and $H(d,w) = -L\mu dw^2\le 0$. Hence, (S1) holds. If $r^*=0$, we see that every feasible point $(d,w)$ to \eqref{prob_reduced_sc} must satisfy $d=r$. Then, in view of \eqref{Htauw}, we can rewrite \eqref{prob_reduced_sc} as:
        \[ \sup_{w>0} -L\mu rw^2.    \] 
        This means that every feasible solution of \eqref{prob_reduced_sc} has a negative objective function value. Then (S1) holds. 
        \item[Case 2:] $r<r^*$. Setting $d=0$ and selecting $w>0$, we see that 
        $$H(d,w)=\left(-L\mu d+(r^*-r+d)\frac{L(L-\mu)^2}{4\mu}\right)w^2=(r^*-r)\frac{L(L-\mu)^2}{4\mu}w^2\geq 0,  $$
        and $r^*-r+d>0$. Then (S2) holds.  
        \item[Case 3:] $m>r\geq r^*$ and $L=\mu$ (\ie, $\kappa = 1$). Then, we have
        \[ H(d,w)=\left(-L\mu d+(r^*-r+d)\frac{L(L-\mu)^2}{4\mu}\right)w^2=-L^2 d w^2.    \]
        Then the optimal value of \eqref{prob_reduced_sc} is $0$, and for all feasible solution $(d,w)$ of \eqref{prob_reduced_sc} with $r^*-r+d>0$ it holds that $d>r-r^*\geq 0$, and hence $H(d,w)<0$. Then (S1) holds. 
        \item[Case 4:] $m>r\geq r^*>0$ and $L>\mu$  (\ie, $\kappa > 1$). If $\kappa=\frac{L}{\mu}=3$, then we have $-L\mu+\frac{L(L-\mu)^2}{4\mu}=L\mu(\frac{(\kappa-1)^2}{4}-1)=0$. Therefore, \eqref{prob_reduced_sc} is equivalent to that:
        \[     \sup_{w\in \R,~d\in \n_0}  (r^*-r)\frac{L(L-\mu)^2}{4\mu}w^2\quad\quad s.t. ~w>0,~  r-r^*\leq d\leq \min\{r,m-r^*\}.    \]
        In this case, we clearly see that the optimal value of \eqref{prob_reduced_sc} is $0$, and is achievable if and only if $r=r^*$. If $r>r^*$, then (S1) holds. If $r=r^*$, then any feasible solution to \eqref{prob_reduced_sc} is optimal. Since $r=r^*<m$, for $\hat d := \min\{r,m-r^*\}$ and any $w > 0$, the point $(w,\hat d)$ is feasible (because $\min\{r,m-r^*\}>0$), and in this case we have $r^*-r+\hat d=\hat d>0$, which means (S2) holds.

        Next we assume $\kappa\neq 3$. Let $\alpha:=\frac{r-r^*}{1-\frac{4}{(\kappa-1)^2}}=\frac{r-r^*}{1-\frac{4\mu^2}{(L-\mu)^2}}$. We now rewrite $H$ in \eqref{Htauw} as:
        \begin{align*}
            &H(d,w)=\left(-L\mu d+(r^*-r+d)\frac{L(L-\mu)^2}{4\mu}\right)w^2 \\
            &=\frac{L(L-\mu)^2}{4\mu}\left(\frac{-4 \mu^2 d}{(L-\mu)^2}+r^*-r+d\right)w^2 =\frac{L(L-\mu)^2}{4\mu}\left(\frac{-4d}{(\kappa-1)^2}+r^*-r+d\right)w^2 \\
            &=\frac{L(L-\mu)^2}{4\mu}\left(\left(1-\frac{4}{(\kappa-1)^2}\right)d+r^*-r\right)w^2 =\frac{L(L-\mu)^2}{4\mu}\left(1-\frac{4}{(\kappa-1)^2}\right)(d-\alpha)w^2.
        \end{align*}
        Then, we can rewrite \cref{prob_reduced_sc} as: 
        \[    
        {\everymath = {\displaystyle}
        \begin{array}{cl}
\sup_{w\in \R,~d\in \n_0}  &\frac{L(L-\mu)^2}{4\mu} \left(1-\frac{4}{(\kappa-1)^2}\right)(d-\alpha)w^2\\
s.t. &w>0,~  r-r^*\leq d\leq \min\{r,m-r^*\}, ~\alpha=\frac{r-r^*}{1-\frac{4}{(\kappa-1)^2}}.
        \end{array}
        }
        \]
        If $\kappa<3$ and $r\geq r^*$, then $\alpha\leq 0$, and we see that the optimal value of \eqref{prob_reduced_sc} is $0$. Moreover, for all feasible solution $(d,w)$ of \eqref{prob_reduced_sc} with $r^*-r+d>0$, we have $d>r-r^*\geq 0$ and $H(d,w)<0$. Then (S1) holds.  

        Finally, we assume $\kappa>3$. If $\alpha>\min\{r,m-r^*\}$, then the optimal value of \eqref{prob_reduced_sc} is $0$, and for all $(d,w)$ that is feasible to \eqref{prob_reduced_sc}, it holds that $H(d,w)<0$. Then (S1) holds. If $\alpha\le\min\{r,m-r^*\}$, then we can select $d=\min\{r,m-r^*\}$ and any $w>0$, which is feasible for \eqref{prob_reduced_sc} and $r^*-r+d=\min\{r^*,m-r\}>0$, and we have $H(d,w)\geq 0$. Then (S2) holds. 
    \end{itemize}
    In summary, note that we have argued that we have either (S1) or (S2). Moreover, (S1) holds if and only if any of the following is true, and (S2) holds otherwise.
        \begin{itemize}
            \item[(1)] $r=m$ or $r^*=0$.
            \item[(2)] $m>r\geq r^* > 0$ and $\kappa = 1$.
            \item[(3)] $m>r\geq r^*>0$, $\kappa > 1$, $\kappa=3$ and $r>r^*$.
            \item[(4)] $m>r\geq r^*>0$, $\kappa > 1$, $\kappa<3$.
            \item[(5)] $m>r\geq r^*>0$, $\kappa > 1$, $\kappa>3$ and $\alpha>\min\{r,m-r^*\}$.
        \end{itemize}
        Upon integrating (1) into the other conditions and regrouping (2), (3) and (4), we can further rewrite the above conditions as follows: 
        \begin{itemize}
            \item[(1)] $r=m$ or $r^*=0$.
            \item[(2)] $r>r^*$ and $\kappa=3$.
            \item[(3)] $r\geq r^*$ and $\kappa<3$.
            \item[(4)] $r\geq r^*$, $\kappa>3$ and $\alpha>\min\{r,m-r^*\}$.
        \end{itemize}
        Next, notice that the condition $r^* = 0$ is not ruled out by (2), (3) and (4); in particular, observe that the condition $\alpha>\min\{r,m-r^*\}$ holds trivially when $\kappa > 3$ and $r^* = 0$, because $r\ge 1$. Thus, we can further rewrite the above conditions as follows: 
        \begin{itemize}
            \item[(1)] $r=m$.
            \item[(2)] $r>r^*$ and $\kappa=3$.
            \item[(3)] $r\geq r^*$ and $\kappa<3$.
            \item[(4)] $r\geq r^*$, $\kappa>3$ and $\alpha>\min\{r,m-r^*\}$.
        \end{itemize}
Finally, we notice that $r-r^*=\min\{r,m-r^*\}-\min\{r^*,m-r\}$, which can be proved by discussing the two cases $r+r^*\leq m$ and $r+r^*> m$ separately. Then, we have $\alpha=(\min\{r,m-r^*\}-\min\{r^*,m-r\})/(1-\frac{4}{(\kappa-1)^2})$. Thus, when $\kappa>3$, we can rewrite the condition $\alpha>\min\{r,m-r^*\}$ as 
\begin{equation}\label{20251226lastlastlast}
\min\{r,m-r^*\}>\frac{(\kappa-1)^2}{4}\min\{r^*,m-r\}.
\end{equation} 
Moreover, when $\kappa = 3$, the condition \cref{20251226lastlastlast} is the same as $\min\{r,m-r^*\}-\min\{r^*,m-r\} > 0$, i.e., $r > r^*$. Furthermore, when $\kappa < 3$ and $r^* \le r < m$, we have $\frac{(\kappa-1)^2}{4} < 1$ and hence we always have
\[
\min\{r,m-r^*\}-\min\{r^*,m-r\} = r - r^* > \left(\frac{(\kappa-1)^2}{4} - 1\right)\min\{r^*,m-r\};
\]
this is because the strict inequality holds trivially when $r > r^*$, and if $r = r^*$, we have $\min\{r^*,m-r\} = \min\{r,m-r\} > 0$, which also implies the strict inequality above. Hence, we see from the above display that \cref{20251226lastlastlast} holds trivially when $\kappa < 3$ and $r^* \le r < m$. Consequently, the above four cases now can be recapped as the following two cases:
        \begin{itemize}
            \item[(1)] $r=m$.
            \item[(2)] $r\geq r^*$ and $ \min\{r,m-r^*\}>\frac{(\kappa-1)^2}{4}\min\{r^*,m-r\} $.
        \end{itemize}
\end{proof}
\revise{Let us emphasize the boundary consequence of \cref{lemma5-7}. If $r\ne m$ and
\[
\min\{r,m-r^*\}=\frac{(\kappa-1)^2}{4}\min\{r^*,m-r\},
\]
then condition (1) in \cref{lemma5-7} fails and condition (2) also fails because its inequality is strict. Therefore the ``otherwise'' part of \cref{lemma5-7} applies: there exists a feasible point of \eqref{eq:opt_sc_introducew} with nonnegative objective value and $\|\bar g\|_\infty>\lambda$. This is the boundary case used to construct the non-$r$-factorizable example in \cref{thm5-9}.}

\bibliographystyle{plain}
 \bibliography{Commonbib}

\end{document}

%% file: sec1.tex
\section{Introduction}
\label{sec:introduciton}
Low-rank matrix estimation has been an extremely important and versatile problem that has attracted intense research over the last two decades and found many applications across a wide range of domains, such as network science~\cite{hsieh2012low}, machine learning~\cite{harchaoui2012large, munson2025introduction}, quantum physics~\cite{liu2011universal}, control~\cite{mohan2010reweighted} and imaging~\cite{zhang2013hyperspectral, hu2018generalized}, to name but a few.
This paper focuses on the following low-rank optimization problem:
\begin{equation}
    \label{prob_nuclear}
    \min_{X\in\R^{m\times n}} f(X):=h(X)+\lambda\|X\|_*,
\end{equation}
where $h$ is assumed to be twice continuously differentiable, $\lambda > 0$, and $\| \cdot\|_*$ denotes the nuclear norm. Without loss of generality, we assume $m\leq n$ throughout the paper.
Problem~\eqref{prob_nuclear} takes the form of the so-called composite minimization which has been heavily studied, especially when $h$ is convex. Therefore, in principle it can be solved by many existing algorithms for composite minimization, including in particular various proximal algorithms~\cite{parikh2014proximal,  lee2014proximal, wen2017linear, yue2019family}, thanks to the closed-form expression of the proximal operator of the nuclear norm~\cite{cai2010singular}.

Nonetheless, in contemporary applications, the dimensions $m$ and $n$ of the decision variable~$X$ can potentially be extremely high, rendering methods working directly with the variable $X$ impractical. For example, in collaborative filtering, which is a classical application of low-rank matrix estimation, the dimensions $m$ and $n$ could be of the order of millions or even higher~\cite{munson2025introduction}. Worse still, the computational cost of the proximal operator associated with the nuclear norm, which is a fundamental building block of many existing algorithms for solving problem~\eqref{prob_nuclear}, is a cubic function in $m$ and $n$, as it involves the singular value decomposition of~$X$. 
To circumvent this, researchers proposed to tackle problem~\eqref{prob_nuclear} via the Burer-Monteiro factorization technique~\cite{burer2003nonlinear, burer2005local, park2017non, zhu2021global, zhang2024improved}, which replaces the variable $X$ by a low-rank approximation $UV^\top$ and solves the resulting problem:
\begin{equation}
    \label{prob_relaxation}
    \min_{U\in\R^{m\times r},V\in\R^{n\times r} }F_r(U,V):=h(UV^\top)+\frac{\lambda(\|U\|_F^2+\|V\|_F^2)}{2},
\end{equation}
where $\|\cdot\|_F$ denotes the Frobenius norm and $r\in [1,m]$ is an integer parameter specified by the modeler. 
The advantage of the factorized problem~\eqref{prob_relaxation} over problem~\eqref{prob_nuclear} is twofold. First, the objective function in the factorized problem~\eqref{prob_relaxation} is differentiable as long as $h$ is. In contrast, the objective function in problem~\eqref{prob_nuclear} is nonsmooth because of the nuclear norm. Second, we often choose $r \ll \min\{m,n\}$ in practice. The total size $r(m+n)$ of the matrix variables $(U,V)$ is therefore substantially smaller than the size $mn$ of the variable $X$ in the non-factorized counterpart~\eqref{prob_nuclear}.

Since our goal is to solve problem~\eqref{prob_nuclear}, the factorization rank $r$ in \eqref{prob_relaxation} cannot be too small. Indeed, optimal solutions of problem~\eqref{prob_nuclear} cannot be recovered by solving problem~\eqref{prob_relaxation} through the correspondence $(U,V)\mapsto UV^\top$ if $r$ is less than the minimum rank $r^*$ of the optimal solutions of problem~\eqref{prob_nuclear}. This issue is currently addressed indirectly as follows. First, it can be readily shown that $U^* {V^*}^\top$ is a global minimizer of problem~\eqref{prob_nuclear} for any global minimizer $(U^*,V^*)$ of problem~\eqref{prob_relaxation} if $r \ge r^*$ (\eg, see~\cite[Lemma 1]{lee2022accelerating}). Second, despite the non-convexity due to the bilinear term $UV^\top$, the factorized problem~\eqref{prob_relaxation} has no spurious local minimizers or second-order stationary points if $r$ is {\em sufficiently large} and $h$ satisfies certain technical conditions~\cite{li2017geometry, zhang2024improved}, which implies that one can actually solve problem~\eqref{prob_nuclear} with a strongly convex $h$ by using any optimization algorithm with a second-order convergence guarantee on problem~\eqref{prob_relaxation}. However, a proper choice of the parameter $r$ is highly nontrivial. As demonstrated by an example constructed in~\cite{zhang2024improved}, merely having $r\ge r^*$ is not enough in general, let alone that the minimum solution rank $r^*$ is often unknown in practice. Our paper also revolves around the choice of the factorization rank $r$ by asking a different but more direct question:
\begin{center}
\fbox{
\begin{minipage}{0.8\textwidth}
\textit{When do all the second-order stationary points of problem~\eqref{prob_relaxation} correspond to \\the global minimizers of problem~\eqref{prob_nuclear} via the mapping $(U,V)\mapsto UV^\top$?}
\end{minipage}}
\end{center}
This motivates the following definition.
\begin{defn}[$r$-factorizability]
    \label{$r$-factorizable}
    Let $h$ be twice continuously differentiable. The function~$f$ in problem~\eqref{prob_nuclear} is said to be $r$-factorizable if every second-order stationary point $(U,V)$ of the function $F_r$ in problem~\eqref{prob_relaxation} satisfies that $UV^\top$ is a global minimizer of $f$.
\end{defn}
\noindent With this definition, our problem is equivalent to the investigation of the $r$-factorizability of the objective function $f$ of problem~\eqref{prob_nuclear}. Most previous works studying the $r$-factorizability, if not all, rely on the restricted isometry property of $h$~\cite{ge2016matrix,ge2017no,zhu2021global,zhang2021general}. Here, we recall that for $\delta >0$ and integers $s,t \ge 0$, a twice continuously differentiable function $h:\mathbb{R}^{m\times n} \to \mathbb{R}$ is said to satisfy $\delta$-$\RIP_{s,t}$ condition~\cite{li2017geometry,zhu2018global,zhang2021general} if for all $X,H\in \R^{m\times n}$ with $\rank(X)\leq s$ and $\rank(H)\leq t$, it holds that
\begin{equation*}
    \label{rip}
    (1-\delta) \|H\|_F^2  \leq  \nabla^2h(X)[H,H]\leq  (1+\delta)\|H\|_F^2.
\end{equation*}

A closely related subject concerns the factorizability in the symmetric, unregularized setting, in which case the analogue problem pairs are
\begin{equation}
    \label{pd_version}
    \min_{X\in {\mathbb{S}}_+^n}    h(X)\ \ \ \ {\rm and}\ \ \ \ \min_{U\in \R^{n\times r}}    \widetilde h(U):= h(UU^\top)
\end{equation}
where $\mathbb S_+^n$ is the set of $n\times n$ symmetric positive semidefinite matrices and $h$ is $C^2$. In \cite{zhang2024improved,zhang2025sharp}, the author considered the case where $\nabla h$ is Lipschitz differentiable and $h$ is strongly convex, and showed that all second-order stationary points $U^*$ of $\widetilde h$ satisfy that $U^*{U^*}^\top$ is the unique minimizer $X^*$ of $h$ over ${\mathbb{S}}_+^n$ under suitable conditions on the factorization rank~$r$, solution rank~$r^*=\rank(X^*)$ and the condition number~$\kappa$ (\ie, the ratio between the Lipschitz constant of $\nabla h$ and the strong convexity modulus of $h$), namely (1) $r\geq r^*$ and $\kappa<3$; or (2) $n > r\geq r^*$ and $r> \tfrac{1}{4}(\kappa-1)^2r^*$; the author also constructed a function~$h$ with $\kappa=3$ such that $\widetilde h$ has a second-order stationary point that does not correspond to any global minimizer of $h$ over~$\mathbb{S}_+^n$. The result was then extended to the class of non-strongly convex quadratic functions satisfying the $\delta$-$\RIP_{r+r^*,r+r^*}$ condition in \cite[Corollary 1.5]{zhang2024improved} or \cite[Theorem 1.4]{zhang2025sharp}, with essentially the same bound on~$\delta$.

The asymmetric, unregularized case has also been studied in the literature, where the factorized problem is
\begin{equation}\label{prob_rank_nonconvex}
   \min_{U\in \R^{m\times r},V\in\R^{n\times r}}  h(UV^\top),
\end{equation}
\revise{The exact second-order landscape of \cref{prob_rank_nonconvex} has been analyzed under RIP-type assumptions; for instance, results in~\cite{ge2017no,zhang2021general} guarantee the nonexistence of spurious second-order stationary points when the RIP constant is sufficiently small. At the same time, \cite[Example~1.8]{zhang2025sharp} shows that, even under the strongest $0$-$\RIP_{m,m}$ condition, unfavorable $\epsilon$-second-order stationary behavior can occur when $r<m$. These results motivate the search for alternative formulations with better landscape properties.} One example is
\begin{equation}
\label{prob_modified}
    \min_{U\in \R^{m\times r},V\in\R^{n\times r}} \tilde F(U,V):= h(UV^\top)+\beta\|U^\top U - V^\top V\|_F^2,
\end{equation}
which was first proposed in \cite{tu2016low}. Upon defining $h_a:\bS^{m+n}\to \bS^{m+n}$ as
\begin{align}
\label{defn_ha}
     h_a(X):=h(X_2)+\beta(\|X_1\|_F^2+\|X_3\|_F^2-2\|X_2\|_F^2)   ,~~X=\begin{bmatrix}
    X_1 & X_2 \\
    X_2^\top & X_3
\end{bmatrix},~X_1\in \bS^{m},
\end{align}
one can see that $\tilde F(U,V)=h_a\left(\begin{bmatrix}
    U \\
    V
\end{bmatrix}\begin{bmatrix}
    U ^\top\!\!&\!\! V^\top
\end{bmatrix}\right)$; in this case, it can be shown that the $\delta$-$\RIP_{k,k}$ condition of $h$ implies the $2\delta$-$\RIP_{k,k}$ condition of $h_a$, see \cite[Fact 3.14]{zhang2025sharp}. Then one can reduce \eqref{prob_modified} to an instance of the second problem in \eqref{pd_version} and apply \cite[Corollary 1.5]{zhang2024improved}. The best known $\RIP$-type condition for ensuring the non-existence of spurious second-order stationary point of \cref{prob_modified} was established in \cite[Theorem 1.6]{zhang2025sharp}.


For the \revise{asymmetric}, regularized problem~\cref{prob_nuclear}, fewer works have been done. A common assumption in these works is that there exists an optimal solution to problem~\eqref{prob_nuclear}  and $r$ is chosen to be at least the minimum solution rank~$r^*$. 
In \cite[Theorem 3]{mcrae2025low}, the author showed that if $h$ is convex quadratic and satisfies the $\delta$-$\RIP_{2r,2r}$ condition with $\delta<\tfrac{1}{3}$, then the corresponding $f$ in problem~\eqref{prob_nuclear} is $r$-factorizable. A similar result was established in~\cite[Theorem 2]{li2017geometry} for a general twice continuously differentiable convex function~$h$ with a more restrictive bound on~$\delta$. Later, in \cite[Theorem 1]{kim2025lora},\footnote{The results in \cite{kim2025lora} were stated in terms of restricted strong convexity and restricted smoothness. The moduli $\alpha$ and $\beta$ therein correspond to $1-\delta$ and $1+\delta$ in our discussion here, respectively.} it was shown that when $h$ satisfies the $\delta$-$\RIP_{2r,2r}$ condition with $\delta<\tfrac{1}{3}$, then the corresponding $f$ in problem~\eqref{prob_nuclear} is $r$-factorizable, and when $\delta\geq \frac{1}{3}$, a second-order stationary point of problem~\eqref{prob_relaxation} corresponds to an approximate stationary point of problem~\eqref{prob_nuclear}. We are not aware of any tight result for \cref{prob_nuclear}, in the sense that once the proposed sufficient conditions fail, then an $h$ can be explicitly constructed with the corresponding $f$ in \eqref{prob_nuclear} being non-$r$-factorizable.


It is tempting to reduce the asymmetric problem to a symmetric problem, as in the unregularized case. Unfortunately, as we shall discuss in~\cref{sec:proof_tech}, this idea is inapplicable in the regularized case. 
To circumvent this, we adopt a different strategy to tackle the factorizability problem, which is outlined in \cref{sec:proof_tech}. Our new techniques enable us to derive tight RIP-type conditions in the sense of \cref{thm5-9} below.

To present our main results, we define the following classes. \revise{Throughout this paper, we assume $\lambda>0$ in \cref{prob_nuclear} is a fixed constant.}
\begin{defn}
\label{defn_class}
   Let $L\in (0,\infty)$, $\mu> 0$, $q,r^*\in [m]\cup\{0\}$. We define $\fS(L,\mu,q,r^*)$ to be the set of all $h\in C^2(\R^{m\times n})$ satisfying the following conditions:
   \begin{enumerate}[label=(\roman*)]
       \item For all $X,H\in \R^{m\times n}$ with $\rank(X),\rank(H)\leq q+r^*$, it holds that:
       \begin{equation}
       \label{eq_RIP}
           \mu \|H\|_F^2 \leq \nabla^2h(X)[H,H]\leq L\|H\|_F^2.
       \end{equation}
       \item There exists a global minimizer $X^*\in \R^{m\times n}$ of $f$ in \eqref{prob_nuclear} satisfying $\rank(X^*)=r^*$.
   \end{enumerate}
\end{defn}
When $L=1+\delta$ and $\mu=1-\delta$, \cref{eq_RIP} can be viewed as the $\delta$-$\RIP_{q+r^*,q+r^*}$ condition. Let us also note that when $q+r^*\geq m$, \cref{eq_RIP} reduces to the $\mu$-strong convexity of $h$ and $L$-Lipschitz continuity of $\nabla h$. To simplify the notation, we denote this latter kind of function classes by $\fS(L,\mu,r^*)$. Our first main result is the following theorem, which characterizes the $r$-factorizability of $\fS(L,\mu,r^*)$ in the following sense: when the conditions in \cref{thm5-9} are satisfied, then for all $h$ in $\fS(L,\mu, r^*)$ the corresponding $f$ in \cref{prob_nuclear} is $r$-factorizable, and if not, there exists $h\in \fS(L,\mu,r^*)$ such that the corresponding $f$ in \cref{prob_nuclear} is not $r$-factorizable.\footnote{\revise{We will refer to this class-level statement as ``tight in the sense of \cref{thm5-9}" for the rest of this paper. The tightness here is interpreted in a worst-case sense over the specified function class: when the stated condition fails, the class contains at least one non-$r$-factorizable example, but this does not assert that every particular function in that parameter regime is non-$r$-factorizable.}}
\begin{thm}
    \label{thm5-9}
    Let $r^*\in [m]\cup\{0\}$, $r\in [m]$, $\infty>L\geq \mu>0$ and $\kappa:=\frac{L}{\mu}\geq 1$. Suppose that $r^*$, $r$ and $\kappa$ satisfy any of the following conditions:
    \begin{enumerate}
            \item[(1)] $r=m$.
            \item[(2)] $r\geq r^*$ and $ \min\{r,m-r^*\}>\frac{(\kappa-1)^2}{4}\min\{r^*,m-r\} $.
    \end{enumerate}
    Then for all $h\in\fS(L,\mu,r^*)$, the corresponding $f$ in \cref{prob_nuclear} is $r$-factorizable. Otherwise, there exists a quadratic $h\in\fS(L,\mu,r^*)$ such that the corresponding $f$ in \cref{prob_nuclear} is not $r$-factorizable.
\end{thm}
Next, by analyzing suitable subspaces, we also obtain the following corollary from Theorem~\ref{thm5-9}. 
\begin{coro}
\label{coro1-4}
    \cref{thm5-9} holds when $\fS(L,\mu, r^*)$ is replaced by $\fS(L,\mu,r,r^*)$.
\end{coro}
In the case of $r+r^* < m$, $L=1+\delta$ and $\mu=1-\delta$, \cref{coro1-4} implies that $f$ is $r$-factorizable if $r/r^*>\frac{\delta^2}{(1-\delta)^2}$ and $h$ satisfies the $\delta$-$\RIP_{r+r^*,r+r^*}$ condition with a global minimizer $X^*\in \R^{m\times n}$ satisfying $\rank(X^*)=r^*$, and that an explicit counterexample can be constructed if $r/r^*\leq \frac{\delta^2}{(1-\delta)^2}$. \revise{Equivalently, letting $\rho:=r/r^*$, the sufficient condition can be rewritten as $\delta<\frac{\sqrt{\rho}}{1+\sqrt{\rho}}$ and the counterexample condition as $\delta\geq\frac{\sqrt{\rho}}{1+\sqrt{\rho}}$. Thus, for fixed $\rho$, a counterexample exists when the RIP constant is at least this threshold.} This bound matches the necessity condition in \cite[Theorem 1.6]{zhang2025sharp}, which studied the different model \cref{prob_modified}. However, we would like to point out that the sufficient condition in \cite[Theorem 1.6]{zhang2025sharp} requires $r/r^*>4\delta^2/(1-2\delta)^2$, which does not match the necessary condition there. 

Interestingly, our bound here also matches the bound in \cite[Theorem 1.1]{zhang2024improved}, which is for \cref{pd_version} though. However, as we will discuss in \cref{sec:proof_tech} below, our results do not follow from \cite{zhang2024improved} and cannot imply the results in \cite{zhang2024improved} and \cite{zhang2025sharp}, since we require the constant $\lambda$ to be positive in \cref{prob_nuclear}, which yields special structures on the first-order stationary points of $F_r$ in \cref{prob_relaxation}, c.f. \cref{stationary_point}.

\subsection{\revise{The break down of symmetric reduction and a novel proof technique}}
\label{sec:proof_tech}

To the best of our knowledge, all the previous works \cite{tu2016low,ge2017no,zhang2021general,zhang2025sharp} on the non-existence of spurious second-order stationary point of asymmetric problem~\cref{prob_modified} rely on symmetric reduction. To study the factorizability of \cref{prob_nuclear}, naturally, one would want to invoke a similar reduction and then invoke results for the symmetric case, such as~\cite{zhang2024improved,zhang2025sharp}. However, such an idea will not work, as we now explain.
First, note that the symmetric case counterpart of \cref{prob_nuclear} is to minimize
\[ h_s(X)=h(X_2)+\frac{\lambda}2\tr(X),~~X=\begin{bmatrix}
    X_1 & X_2 \\
    X_2^\top & X_3
\end{bmatrix}.    \]
We therefore have $h_s\left(\begin{bmatrix}
    U \\
    V
\end{bmatrix}\begin{bmatrix}
    U^\top \!&\! V^\top
\end{bmatrix}\!\right)=F_r(U,V)$. However, $h_s$ loses the $\RIP$ property by \cref{fact6-1}, since $h_s$ is not even strongly convex on the one-dimensional linear subspace $\{t E_{11}:~t\in \R\}$, where \revise{$E_{11}\in \R^{(m+n)\times(m+n)}$} is the matrix whose $(1,1)$-entry is $1$ and is zero otherwise. 
\revise{This causes the break down of the PSD Burer--Monteiro arguments.}

\revise{In \cite[Lemma~3.2]{zhang2024improved}, the problem of constructing a counterexample  with a fixed minimizer and a fixed spurious point is reformulated as an SDP whose constraints involve the Hessian operator of the lifted loss, in particular its action on the symmetric factorization directions through terms such as $\mathcal J_X^\top \mathcal H\mathcal J_X$. In \cite{zhang2025sharp}, the corresponding escape-direction proof is built from the sharp directions in \cite[Definition~3.1]{zhang2025sharp}, the tangent-normal decomposition in \cite[Fact~3.2]{zhang2025sharp}, and the resulting symmetric escape certificate in \cite[Theorem~3.10]{zhang2025sharp}. These certificates require curvature or RIP control of the lifted loss along low-rank symmetric tangent and normal directions generated by the PSD factorization. Such directions are not confined to the off-diagonal block $X_2$; they also contain diagonal-block variations in $X_1$ and $X_3$. For the lifted function $h_s$ above, however, the original loss $h$ only controls the off-diagonal block, and the trace term $\frac{\lambda}{2}\tr(X)$ is linear and hence does not affect the Hessian. Thus, the curvature certificate used in the PSD SDP or tangent-normal escape argument is not applicable to $h_s$.}

In contrast, our proofs do not rely on the symmetric reduction technique. \revise{The role played by the PSD tangent-normal structure is replaced by the first-order structure induced by the positive regularization parameter $\lambda$. Specifically, the first-order equations of $F_r$ contain the terms $\lambda U$ and $\lambda V$, and these terms force the balancedness condition $U^\top U=V^\top V$. This balancedness is then used in \cref{stationary_point} to show that $UV^\top$ and $-\nabla h(UV^\top)$ admit compatible singular-vector bases, with the ``active" singular values of $-\nabla h(UV^\top)$ equal to $\lambda$.} We start with the following well-known fact, essentially a restatement of \cite[Theorem 2.1.12]{nesterov2018lectures}: if $h\in C^1(\R^{m\times n})$ is $\mu$-strongly convex and $\nabla h$ is $L$-Lipschitz continuous, then for all $X,Y\in \R^{m\times n}$ it holds that
\begin{equation}
    \label{base_ineq}
    (L-\mu)\langle  \nabla h(X)-\mu X-(\nabla h(Y)-\mu Y), X-Y     \rangle\geq  \|\nabla h(X)-\mu X -(\nabla h(Y)-\mu Y)\|_F^2 .
\end{equation}
Next, by analyzing the first-order stationary point of $F_r$ in \cref{prob_relaxation}, we show that if $(\bar U,\bar V)$ is a stationary point of $F_r$, then $\bar X=\bar U\bar V^\top$ is a pseudo stationary point of $f$, in which case $\bar X$ and $\nabla h(\bar X)$ admit simultaneous singular value decompositions;  see \cref{stationary_point} and \cref{defn-psp}. Using a variant of Von Neumann's trace inequality (see \cref{prop2-6}), we transform \cref{base_ineq} into a bound concerning the singular values of $\bar X,X^*,\nabla h(\bar X),\nabla h(X^*)$, when $X^*$ and $\bar X$ are pseudo stationary points of $f$. Then, the existence of counterexamples can be transformed to the existence of a certain feasible solution to a quadratic program involving $L$ and $\mu$; see \cref{eq:opt_sc}. We then solve this optimization problem analytically, which yields tight bounds in the sense of \cref{thm5-9}.

\revise{Let us also clarify the scope of this proof strategy. For a general regularizer, one must first specify the associated factorized model before asking whether the present factorization-based argument extends. In the unweighted nuclear-norm case, the problem pair \cref{prob_nuclear}--\cref{prob_relaxation} is canonical because $\|X\|_*=\min_{X=UV^\top}\frac{1}{2}(\|U\|_F^2+\|V\|_F^2)$. Thus $h(X)+\lambda\|X\|_*$ naturally leads to $h(UV^\top)+\lambda(\|U\|_F^2+\|V\|_F^2)/2$. The proof then uses features that are specific to this pair: the isotropic Frobenius regularizer gives first-order optimality conditions with the same scalar $\lambda$ in the $U$- and $V$-blocks, which yield balancedness, the simultaneous singular-vector structure in \cref{stationary_point}, and the constant active/inactive threshold in the nuclear-norm subdifferential. These scalar constraints are what make the singular-value reduction and the finite-dimensional quadratic program possible. Other regularizers may not have the corresponding factorization models. For example, it is unclear whether weighted nuclear norms or rank-constrained formulations correspond to a useful factorized models. Some special Schatten quasi-norm regularizers do have natural factorized representations; for example, a three-factor model of the form $h(UVW)+\|U\|_F^2+\|V\|_F^2+\|W\|_F^2$ corresponds, after eliminating the factors, to a Schatten-$2/3$ quasi-norm penalty proportional to $\sum_i\sigma_i(X)^{2/3}$. However, this produces a nonconvex regularized problem. In such a setting, first-order stationary points of the non-factorized objective need not be global minimizers, so the factorizability question studied in this paper would have to be reformulated. Extensions to these models therefore require model-specific landscape certificates rather than a direct adaptation of the present proof.}

The generalization to the RIP case is based on the observation that, when \cref{eq_RIP} holds, the function $h$ is $\mu$-strongly convex and $\nabla h$ is $L$-Lipschitz continuous on any linear subspace of $\R^{m\times n}$ consisting of merely matrices whose rank is no more than $q+r^*$; see \cref{fact6-1}. In particular, for $\rank(X^*)=r^*$ and $\rank(\bar X)=r$, we can restrict the function on a linear subspace of $\R^{(r+r^*)\times (r+ r^*)}$ under proper orthogonal transformation, which contains $X^*$ and $\bar X$ (see \cref{common_svd}), and then apply the result for the strongly convex case (see \cref{prop_rip}). 

\revise{We also comment on the role of the $C^2$ assumption. The restricted RIP-type condition \cref{eq_RIP} is stated in Hessian form because this paper works in the classical smooth setting. For the singular-value and quadratic-program analysis after the reduction, however, one could replace this pointwise Hessian condition by a restricted secant condition of the form $\mu\|X-Y\|_F^2\leq \langle\nabla h(X)-\nabla h(Y),X-Y\rangle\leq L\|X-Y\|_F^2$ for matrices $X,Y$ in the relevant rank-restricted set. Once the singular-value information for second-order stationary points is available, the arguments in \cref{boundr} and \cref{sec:gen_rip} use only such first-order curvature information, pseudo-stationarity, and the finite-dimensional quadratic program. The main place where $C^2$ is essential in the present proof is \cref{prop4-7}: there we derive the bound on the singular values of $-\nabla h(\bar U\bar V^\top)$ that correspond to zero singular value of $\bar U\bar V^\top$ from the classical Hessian-based second-order stationarity of $F_r$. Thus, removing the $C^2$ assumption would require choosing an appropriate generalized notion of second-order stationarity and proving an analogue of \cref{prop4-7} under that notion. We do not pursue this separate generalized second-order analysis here.}

Let us now compare our proof techniques to those used in the literature. In~\cite{zhang2024improved,zhang2025sharp}, tight bounds on the $\RIP$ constant $\delta$ analogous to those in our \cref{thm5-9} were obtained for \cref{pd_version}, and their bounds are the same as ours despite the difference in the problems under study. Then the result in \cite{zhang2025sharp} for \cref{pd_version} was generalized to \cref{prob_modified} by using symmetric reduction. Note that the best bounds on the $\RIP$ constant $\delta$ in \cite[Theorem 1.6]{zhang2025sharp} for \cref{prob_modified} are also not tight as in our Theorem~\ref{thm5-9}, which may be attributed to their use of the reduction to the symmetric case, so that the $\delta$-$\RIP_{k,k}$ condition of $h$ only yields the $2\delta$-$\RIP_{k,k}$ condition on $h_a$ in \cref{defn_ha}. Although the tight bound in \cite{zhang2024improved,zhang2025sharp} for \cref{pd_version} aligns with our bound, their techniques are inapplicable in our context. In \cite{zhang2024improved}, the existence of counterexamples was formulated as an semidefinite programming (SDP) problem, then the author solved the optimization problem analytically to get a tight RIP constant bound. In \cite{zhang2025sharp}, the author showed that the existence of (quadratic) counterexamples can be equivalently formulated as the existence of escape direction, and a sharp escape direction is constructed explicitly based on the tangent-normal decomposition of the manifold of $\rank$-$r$ positive semidefinite matrices. In contrast, we formulate a quadratic programming problem and solve it analytically. Our proof techniques are also different from those in \cite{li2017geometry,kim2025lora,mcrae2025low}, which focus on \cref{prob_relaxation}. Indeed, our proof is based on analyzing the singular values of the global minimizer $X^*$ appeared in \cref{defn_class} and other related matrices (see \cref{nprop5-4}), which leads to the tight bounds in the sense of \cref{thm5-9}.

The remainder of the paper is organized as follows. In \cref{sec:notation}, we define the notation and present some preliminary results. The characterization of first- and second-order stationary points of problem~\eqref{prob_relaxation} is presented in \cref{sec:stationary_points}.  In \cref{boundr}, we prove \cref{thm5-9} by using the results in \cref{sec:notation} and \cref{sec:stationary_points}. The generalization to the $\RIP$ case, namely \cref{coro1-4}, is proved in \cref{sec:gen_rip}.

\section{Notation and preliminaries}
\label{sec:notation}
Throughout this paper, we assume that $1\le r\le m\le n$ in problem~\eqref{prob_relaxation}.
For a matrix $X\in\R^{m\times n}$, we let $\|X\|_*$, $\|X\|_2$ and $\|X\|_F$ denote its nuclear norm, spectral norm and Frobenius norm, respectively. The $i$-th largest singular value of $X$ is denoted by $\sigma_i(X)$ for $i = 1,\dots,m$. 
The vector of singular values is denoted by $\sigma(X) = \begin{bmatrix}\sigma_1 (X)&\cdots& \sigma_m(X)\end{bmatrix}^\top$. \revise{For a general matrix $Z\in\R^{p\times q}$, $\sigma(Z)$ is defined analogously as the vector of its $\min\{p,q\}$ singular values.}
The set of $n\times n$ orthogonal matrices is denoted by $\cO^n$. 
For $x\in \R^s$, we denote by $\Diag(x)\in \R^{s\times s}$ the diagonal matrix with $(\Diag(x))_{ii} = x_i$ for $i = 1,\dots,s$. Moreover, we define $\diag:\R^{s\times s}\to \R^s$ to be the adjoint operator of $\Diag$. In this paper, to simplify the presentation, we also use $\wDiag$ and $\wdiag$ to denote the possibly non-square versions of $\Diag$ and $\diag$, respectively. Specifically, for $x\in \R^s$, $\wDiag(x)$ would be a diagonal matrix whose diagonal part is $x$, which is not necessarily square; the dimension of $\wDiag (x)$ is omitted when it can be understood from the context.\footnote{ 
For example, if $R\in \R^{m\times m}$ and $P\in \R^{n\times n}$, then writing $R\,\wDiag(x)P^\top $ would imply that $\wDiag(x)\in \R^{m\times n}$.} Also, for $X=[X_1~~X_2]\in \R^{m\times n}$ with $X_1\in \R^{m\times m}$ and $X_2\in \R^{m\times (n-m)}$, we define $\wdiag(X):=\diag(X_1) \in \R^m$. \revise{For a matrix $Z\in \R^{p\times q}$, we define}
\[
\revise{\cO_Z := \{ (R, P)\in \cO^p\times \cO^q: R\,\wDiag(\sigma(Z))P^\top = Z \}},
\]
\revise{where $Z\in\R^{p\times q}$ and $\wDiag(\sigma(Z))$ is understood as a $p\times q$ matrix.}

For a mapping $H:\R^{m\times n}\to \R^{m\times n}$, we say $H$ is Lipschitz continuous with modulus $L$ if the following holds:
\[ 
\|H(X)-H(Y)\|_F\leq L\|X-Y\|_F\quad \quad \forall X,Y\in \R^{m\times n}.   \]
The strong convexity for an $h\in C^2(\R^{m\times n})$ is also defined with respect to the Frobenius norm. Namely, $h\in C^2(\R^{m\times n})$ is said to be $\mu$-strongly convex if $\nabla^2 h(X)[Y,Y]\geq\mu\|Y\|_F^2$ for all $X,Y\in \R^{m\times n}$, where the Hessian $\nabla^2h(X):\R^{m\times n}\times \R^{m\times n}\to \R$ is regarded as a quadratic form on $\R^{m\times n}$. To avoid clutter, we sometimes use the notation $\nabla^2h(X)[Y]^2$ to denote $\nabla^2h(X)[Y,Y]$. 

 The set of nonnegative integers is denoted by $\n_0$. For a nonnegative integer $r$, we use $[r]$ to denote the set $\{1,\dots,r\}$; in particular, $[0] := \emptyset$. The permutation group of order $m$ is denoted by $\gp_m$, and the set of $m\times m$ permutation matrices is denoted by $\sp_m$. Finally, for an $x\in \R$, we let $\lfloor  x\rfloor$ denote the largest integer upper bounded by $x$. 

We will need the following characterization of the subdifferential of the nuclear norm.
\begin{prop}[{\cite[Example 2]{watson1992characterization}}]
    \label{subd_nuclear}
    Let $X\in \R^{m\times n}$ be a matrix of rank $s$ and $(R,P)\in \cO_X$. Then, 
    \[    \partial \|X\|_* =\left\{ R \begin{bmatrix}  I & 0 \\
    0 & W   \end{bmatrix}P^\top :~W\in \R^{(m-s)\times (n-s)},~\|W\|_2\leq 1  \right\}.       \]
\end{prop}
Note that while the singular value decomposition of $X$ is not unique, the subdifferential $\partial \|X\|_*$ is independent of the choice of the singular value decomposition.

Before ending this section, we present a variant of von Neumann's trace inequality. Its proof requires the following well-known result concerning doubly stochastic matrices.
\begin{lemma}
    \label{lem2-5}
    Let $A\in\R^{m\times  m}$ be a nonnegative matrix that satisfies
    \[ 
    \forall i\in [m],\quad \sum_{j=1}^m A_{ij}\leq 1,\quad \sum_{j=1}^m A_{ji}\leq 1.    
    \]
    Then, there exists a doubly stochastic matrix $B$ such that $B_{ij}\geq A_{ij}$ for all $i$ and $j$.
\end{lemma}
\begin{proof}
    Let ${\cal R}$ and ${\cal C}$ be the sets consisting of the indices of the rows and columns of $A$ whose sum is less than $1$, respectively. Clearly, ${\cal R}$ and ${\cal C}$ must be simultaneously empty or nonempty. We modify the matrix $A$ gradually in the following manner: at each step, we select $i\in {\cal R}$ and $j\in {\cal C}$, and enlarge $A_{ij}$ until either the row sum of $i$-th row or the column sum of $j$-th column reaches $1$. Then we update ${\cal R}$ and ${\cal C}$ and repeat this process. Since ${\cal R}$ and ${\cal C}$ are always simultaneously empty or nonempty, our algorithm is well defined. Moreover, after each step, the quantity $|{\cal R}|+|{\cal C}|$ is reduced by at least $1$. Since this number is finite, we must end with ${\cal R}={\cal C}=\emptyset$. Then the resulting matrix, denoted by $B$, is doubly stochastic, and it holds by construction that $B_{ij}\geq A_{ij}$ for all $i$ and $j$.
\end{proof}
Below is the announced variant of von Neumann's trace inequality, which reduces to the classical von Neumann's inequality when $C$ or $D$ is a zero matrix.
\begin{lemma}
    \label{prop2-6}
    Let $A$, $B$, $C$ and $D$ be nonnegative $m\times m$ diagonal matrices with diagonal vectors $d^A$, $d^B$, $d^C$ and $d^D$, respectively. Then, we have 
    \begin{equation}
        \sup_{R\in \cO^m,P\in \cO^n} \tr(R\begin{bmatrix}
            A & 0
        \end{bmatrix}P \begin{bmatrix}
             B \\
             0
        \end{bmatrix})+ \tr(R\begin{bmatrix}
            C & 0
        \end{bmatrix}P \begin{bmatrix}
             D \\
             0
        \end{bmatrix})  =\max_{E\in \sp_m} (d^A)^\top E d^B+ (d^C)^\top E(d^D),
    \end{equation}   
    where $\sp_m$ is the set of $m\times m$ permutation matrices.
\end{lemma}
\begin{proof}
    For any $R\in \cO^m$ and $P\in \cO^n$, we have
    \begin{align*}
        &\tr(R\begin{bmatrix}
            A & 0
        \end{bmatrix}P \begin{bmatrix}
             B \\
             0
        \end{bmatrix})+ \tr(R\begin{bmatrix}
            C & 0
        \end{bmatrix}P \begin{bmatrix}
             D \\
             0
        \end{bmatrix})\\
        &=\sum_{i,j=1}^m (d^A_id^B_j+d_i^Cd^D_j)P_{ij}R_{ji} \leq \sum_{i,j=1}^m (d^A_id^B_j+d_i^Cd^D_j)(\frac{R_{ji}^2}{2}+\frac{P_{ij}^2}{2}) \\
        &\overset{(\rm a)}{=} \sum_{i,j=1}^m (d^A_id^B_j+d_i^Cd^D_j)Z_{ij}=(d^A)^\top Z d^B+(d^C)^\top Zd^D,
    \end{align*}
    where in (a) we define $Z\in \R^{m\times m}$ such that $Z_{ij}=\frac{R_{ji}^2}{2}+\frac{P_{ij}^2}{2}$ for all $i$ and $j$. Since $R\in\cO^m$ and $P\in\cO^{n}$, we see that all row sums and column sums of $Z$ are at most 1. By \cref{lem2-5}, we know there is a doubly stochastic matrix $Y$ such that $Y_{ij}\geq Z_{ij}$ for all $i$ and $j$. Since $d^A,d^B,d^C,d^D$ are all nonnegative, we have 
    \[   
    (d^A)^\top Z d^B+(d^C)^\top Zd^D \leq (d^A)^\top Yd^B+(d^C)^\top Yd^D.       
    \]
    Applying Birkhoff theorem (see, \eg, \cite[Theorem 1.2.5]{borwein2006convex}), the matrix $Y$ is a convex combination of permutation matrices, namely, $Y=\sum_{i=1}^s\lambda_iP_i$, where $P_i\in\sp_m$, $\lambda_i\geq 0$ for each $i = 1,\ldots,s$ with $\sum_{i=1}^s\lambda_i=1$. Therefore, we see that 
    \[ (d^A)^\top Yd^B+(d^C)^\top Yd^D=\sum_{i=1}^s \lambda_i[(d^A)^\top P_id^B+(d^C)^\top P_id^D]\leq \sup_{E\in\sp_m} (d^A)^\top E d^B+ (d^C)^\top E(d^D). \]
    This upper bound can be achieved by setting $R=E_*^\top$ and $P=\begin{bmatrix} E_* & 0 \\ 0 & I_{n-m} \end{bmatrix}\in \R^{n\times n}$, where $E_*$ achieves the supremum in $\sup_{E\in\sp_m} (d^A)^\top E d^B+ (d^C)^\top E(d^D)$.
\end{proof}

%% file: sec3.tex
\section{First- and second-order stationary points of $F_r$}
In this section, we present characterizations of first- and second-order stationary points of $F_r$ in problem~\eqref{prob_relaxation}. Let us note that in the literature, the stationarity of $F_r$ has already been studied in \cite{li2017geometry,li2019non}. However, for our purpose, we need to extract useful features that have not been documented in the literature. To begin with, we can derive directly from the first-order optimality condition of problem \eqref{prob_relaxation} (see \eqref{stat_F} below; see also \cite[Proposition~2]{li2017geometry}) that if $(U,V)$ is a stationary point of $F_r$, then $U^\top U=V^\top V$, which is also called a balanced pair. The next lemma studies how the balancedness can affect the singular vectors of $U$ and $V$, which serves as the basis to establish that $X$ and $\nabla h(X)$ have SVDs with common orthonormal bases (possibly in different order) when $X=UV^\top$ with $(U,V)$ being a first-order stationary point of $F_r$.

\label{sec:stationary_points}
\begin{lemma}
    \label{lem3-1}
    Let $V\in \R^{n\times r}$ and $U\in \R^{m\times r}$. Then, $U^\top U = V^\top V$ if and only if $\sigma(V) = \sigma(U)$ and for any $(P, Q)\in \cO_V$ there exists $R$ such that $(R, Q)\in \cO_U$.
\end{lemma}
\begin{proof}
    To prove the ``if'' direction, we note that by the definitions of $\cO_V$ and $\cO_U$, $P\wDiag(\sigma(V))Q^\top$ and $R\wDiag(\sigma(U))Q^\top$ are singular value decompositions of $V$ and $U$, respectively. Then,
    \begin{align*}
        &\, U^\top U = Q\wDiag(\sigma(U))^\top R^\top R\wDiag(\sigma(U))Q^\top = Q \wDiag ( \sigma(U))^\top \wDiag ( \sigma(U)) Q^\top \\
        = &\, Q \wDiag ( \sigma(V))^\top \wDiag ( \sigma(V)) Q^\top = Q \wDiag ( \sigma(V))^\top P^\top P \wDiag ( \sigma(V)) Q^\top = V^\top V.
    \end{align*}
    
    We next prove the ``only if'' direction. Suppose that $U^\top U = V^\top V$. The equality $\sigma(U) = \sigma (V)$ follows directly from the definition of singular values. For the remaining assertion, let $(P,Q)\in \cO_V$. Then, $V = P\wDiag(\sigma(V))Q^\top$ is a singular value decomposition. By the supposition $U^\top U = V^\top V$, 
    \begin{equation}\label{svd_uo}
    \begin{split}
        Q^\top U^\top U Q &= Q^\top V^\top V Q =  Q^\top (P \wDiag(\sigma(V)) Q^\top)^\top (P \wDiag(\sigma(V)) Q^\top) Q \\
         &= \Diag(\sigma_1^2(V), \dots, \sigma_s^2 (V), 0,\dots, 0),
    \end{split}
    \end{equation}
    where $s := \rank (V)$ and hence $\sigma_1 (V),\dots, \sigma_s (V) > 0$. 
    Denote by $\hat u_i$ the $i$-th column of $UQ$ for $i \in \revise{[r]}$. It then follows from \eqref{svd_uo} that the vectors $\hat u_{1}/\sigma_1 (V) ,\dots,  \hat u_s/\sigma_s(V)$ are orthonormal and that $\hat u_i = 0$ for $i = s+1,\dots, \revise{r}$. There must exist $m-s$ vectors $r_{s+1},\dots, r_m$ so that
    $R = [ \hat u_{1}/\sigma_1 (V) ,\dots,  \hat u_s/\sigma_s(V), r_{s+1}, \dots, r_{m}] \in \cO^m$. By the definition of $R$ and the fact that $\sigma(V) = \sigma(U)$, we have
    \begin{align*}
        R \wDiag(\sigma(U)) Q^\top & = [ \hat u_{1}/\sigma_1 (V) ,\dots,  \hat u_s/\sigma_s(V), r_{s+1}, \dots, r_{m}]\, \wDiag(\sigma_1 (V),\dots, \sigma_s (V), 0,\dots, 0)\, Q^\top \\
        & = [\hat u_{1},\dots, \hat u_s,0,\dots,0]\, Q^\top = UQ Q^\top = U,
    \end{align*}
    which implies $(R,Q)\in \cO_U$ and thus completes the proof.
\end{proof}
Next, we are ready to establish the SVD structures of $X$ and $\nabla h(X)$ when $X=UV^\top$ with $(U,V)$ being a first-order stationary point of $F_r$. The proof is basically done by substituting the SVDs of $(U,V)$ into the first-order optimality condition and then solving the stationary equations. 
\begin{prop}[First-order stationarity]
    \label{stationary_point}
    A pair $(U,V)\in \R^{m\times r}\times \R^{n\times r}$ is a stationary point of $F_r$ in \eqref{prob_relaxation} if and only if there exist $R\in \cO^m$, $P\in \cO^n$ and $Q\in \cO^r$ such that $(R, Q)\in \cO_U$, $ (P, Q)\in \cO_V $, $\sigma (U) = \sigma (V)$, and $ \nabla h( UV^\top ) = - R \,\wDiag(d)\, P^\top $ for some $d\in \R^m$ satisfying $d_1 = \cdots = d_s = \lambda$ and $ d_{s+1} \ge \cdots \ge d_m \ge 0 $, where $s = \rank (U) = \rank(V)$.
\end{prop}
\begin{remark}\label{haha}
\begin{enumerate}
    \item[(i)] Note that the decomposition $ - \nabla h( UV^\top ) =  R \,\wDiag(d)\, P^\top $ in \cref{stationary_point} is not a singular value decomposition in general because it is possible that $d_{s+1} > \lambda = d_1 =\cdots = d_s$. Nevertheless, the vector $d$ contains all the singular values of $ - \nabla h( UV^\top ) $, \ie, $d$ is $\sigma( - \nabla h( UV^\top ) )$ up to a permutation of the entries.
    \item[(ii)]  For a stationary point $(U, V)$ of $F_r$, \cref{stationary_point} shows that $\rank( U)=\rank( V)$ and $U V^\top=R[\wDiag(\sigma(U))]^2P^\top = R[\wDiag(\sigma(V))]^2P^\top$ for some $R\in \cO^m$ and $P\in \cO^n$. Hence, $\sigma_i(UV^\top )=\sigma_i^2(  U)=\sigma_i^2( V)$ for all $i\in [m]$.
\end{enumerate}
\end{remark}
\begin{proof}[Proof of \cref{stationary_point}]
The first-order optimality condition of problem~\eqref{prob_relaxation} reads 
   \begin{equation}
    \label{stat_F}
    \begin{cases}
        \nabla h(UV^\top) V+\lambda U=0, \\
        \nabla h(UV^\top)^\top U+\lambda V=0.
    \end{cases}
   \end{equation}
   We first prove the ``if'' direction. By supposition, we have that 
   \begin{align*}
       &\, \nabla h (UV^\top) V +\lambda U = - R \wDiag(d) P^\top P\wDiag (\sigma(V))Q^\top + \lambda R \wDiag(\sigma(U) ) Q^\top \\
       = &\, - R \wDiag(d) P^\top P\wDiag (\sigma(U))Q^\top + \lambda R \wDiag(\sigma(U) ) Q^\top \\
       = &\, - R \begin{bmatrix}
           \lambda I_s & 0 \\
           0 & \wDiag(d_{s+1}, \dots, d_m) \\
       \end{bmatrix} \begin{bmatrix}
           \Diag (\sigma_1 (U), \dots, \sigma_s (U)) & 0 \\
           0 & 0
       \end{bmatrix} Q^\top + \lambda R \wDiag(\sigma(U) ) Q^\top \\
       = &\, - \lambda R \wDiag(\sigma(U) ) Q^\top + \lambda R \wDiag(\sigma(U) ) Q^\top = 0,
   \end{align*}
   which shows the first equality in \eqref{stat_F}. Similarly, we have
   \begin{align*}
       &\, \nabla h (UV^\top)^\top U +\lambda V = - P \wDiag(d) R^\top R \wDiag(\sigma(U) ) Q^\top + \lambda P\wDiag (\sigma(V))Q^\top \\
       = &\, - P \wDiag(d) R^\top R \wDiag(\sigma(V) ) Q^\top + \lambda P\wDiag (\sigma(V))Q^\top \\
       = &\, - P \begin{bmatrix}
           \lambda I_s &  0\\
           0 & \wDiag(d_{s+1}, \dots, d_m) 
       \end{bmatrix} \begin{bmatrix}
           \Diag (\sigma_1 (V), \dots, \sigma_s (V)) & 0 \\
           0 & 0 
       \end{bmatrix} Q^\top + \lambda P \wDiag(\sigma(V) ) Q^\top \\
       = &\,  - \lambda P \wDiag(\sigma(V) ) Q^\top + \lambda P \wDiag(\sigma(V) ) Q^\top = 0,
   \end{align*}
   which shows the second equality in \eqref{stat_F}. This proves the ``if'' direction. 

   To prove the ``only if'' direction, we assume that $( U, V)$ is a stationary point of $F_r$ in \eqref{prob_relaxation}, \ie, \cref{stat_F} holds. A direct computation shows that $U^\top U = V^\top V$; see also \cite[Proposition~2]{li2017geometry}. Fix a singular value decomposition of $V$:
   \begin{equation}
        \label{svd_v}
        V=P_1\begin{bmatrix} \Diag(\sigma_1 (V), \dots, \sigma_s (V)) & 0\\ 0  & 0\end{bmatrix}Q_1^\top.
    \end{equation}
    By \cref{lem3-1}, there exists some $R_1\in \cO^m$ such that
    \begin{equation}
    \label{svd_u}
        U=R_1\begin{bmatrix} \Diag(\sigma_1 (V), \dots, \sigma_s (V)) & 0 \\ 0 & 0\end{bmatrix}Q_1^\top.
    \end{equation}
    Next, we write
    \begin{equation}
        \label{proof:1}
        \nabla h(UV^\top) = R_1 
    \begin{bmatrix}
        A & B\\ C & D
    \end{bmatrix} P_1^\top ,
    \end{equation}
    for some $A\in \R^{s\times s}$, $B\in \R^{s\times (n-s)}$, $C\in \R^{(m-s)\times s}$, $D\in \R^{(m-s)\times (n-s)}$.
    Substituting \eqref{svd_v}, \eqref{svd_u} and \eqref{proof:1} into \eqref{stat_F}, we get
    \begin{align*}
        A\,\Diag(\sigma_1 (V),\dots, \sigma_s(V)) + \lambda\, \Diag(\sigma_1 (V),\dots, \sigma_s(V)) & = 0, \\
        C\,\Diag(\sigma_1 (V),\dots, \sigma_s(V)) & = 0, \\
        A^\top\, \Diag(\sigma_1 (V),\dots, \sigma_s(V)) +\lambda\, \Diag(\sigma_1 (V),\dots, \sigma_s(V)) & = 0, \\
        B^\top \, \Diag(\sigma_1 (V),\dots, \sigma_s(V)) & = 0,
    \end{align*}
    which imply that $B = C = 0$, $A = -\lambda I_s$ and $D$ is unconstrained.
    
    Finally, let $(R_2, P_2) \in \cO_{-D}$ and define the following orthogonal matrices
    \[
        P=P_1\begin{bmatrix}
        I_s & 0\\
        0 & P_2
       \end{bmatrix}\quad\text{and}\quad R=R_1\begin{bmatrix}
        I_s & 0 \\
        0 & R_2
       \end{bmatrix}.
    \]
    Using \cref{svd_v} and \cref{svd_u}, one can check readily that $(P,Q_1)\in \cO_V$ and $(R,Q_1)\in \cO_U$. Moreover, using \cref{proof:1} together with the facts that $B = C = 0$ and $A = -\lambda I_s$, we see that
    \[
    \nabla h(UV^\top) = - R_1 
    \begin{bmatrix}
        \lambda I_s & 0\\ 0 & -D
    \end{bmatrix} P_1^\top = - R \begin{bmatrix}
        \lambda I_s & 0\\ 0& \wDiag(\sigma (-D))
    \end{bmatrix} P^\top.
    \]
    This completes the proof.
\end{proof}


In view of the first-order optimality condition of \cref{prob_nuclear} and \cref{subd_nuclear}, one can see that the SVD structures of $UV^\top$ and $\nabla h(UV^\top)$ given by \cref{stationary_point} share some features of the true stationary point $X^*$ of $f$ in \cref{prob_nuclear} and the corresponding $\nabla h(X^*)$, even though $UV^\top$ is not necessarily a stationary point of $f$. To treat such points and the true stationary points under a unified framework, we define the following category of points of the function $f$.
\begin{defn}[Pseudo-stationarity]
    \label{defn-psp}
    A matrix $X\in \R^{m\times n}$ is said to be a pseudo-stationary point of $f$ in \eqref{prob_nuclear} if there exist $(R, P)\in \cO_X$ and $d\in \R^m_+$ such that $ -\nabla h(X)= R\wDiag(d) P^\top $ and $d_1 = \cdots = d_s = \lambda$, where $s=\rank(X)$.
\end{defn}
Assume that $X^*$ and $\bar X$ are two pseudo-stationary points of \revise{$f$}. When $h$ is $\mu$-strongly convex and $\nabla h$ is $L$-Lipschitz continuous, we know that \cref{base_ineq} holds with $X= X^*$ and $Y = \bar X$. We next leverage \cref{stationary_point}, which states that $X^*$ and $\nabla h(X^*)$, as well as $\bar X$ and $\nabla h(\bar X)$, admit simultaneous SVDs, to transform \cref{base_ineq} to a bound involving only the singular values of $X^*$, $\nabla h(X^*)$, $\bar X$ and $\nabla h(\bar X)$.  
\begin{lemma}
    \label{prop5-4}
    Let $X_1,X_2$ be two pseudo-stationary points of $f$ in \cref{prob_nuclear} in the sense of Definition~\ref{defn-psp}, \ie, there exist $R_1$, $R_2\in \cO^m$ and $P_1$, $P_2\in \cO^n$ such that
    \begin{equation}
        \label{eq:X12G12}
        \begin{aligned}
            &X_1=R_1\begin{bmatrix}
                \Sigma_1 &  0_{m\times (n-m)}
            \end{bmatrix}   P_1^\top,~ -\nabla h(X_1)=R_1\begin{bmatrix}
                D_1 &  0_{m\times (n-m)}
            \end{bmatrix}P_1^\top, \\
            &X_2=R_2\begin{bmatrix}
                \Sigma_2 & 0_{m\times (n-m)}
            \end{bmatrix}P_2^\top,~ -\nabla h(X_2)=R_2\begin{bmatrix}
                D_2 & 0_{m\times (n-m)}
            \end{bmatrix}P_2^\top, 
        \end{aligned} 
    \end{equation}
    where $\Sigma_i=\diag(\sigma_1(X_i),\dots,\sigma_m(X_i))\in \R^{m\times m}$ and $D_i=\diag(d_1^i,\dots,d_m^i)\in \R^{m\times m}_+$ with $d^i_1=\dots=d^i_{\rank(X_i)}=\lambda$ for $i=1,2$. Assume that $h$ in \cref{prob_nuclear} satisfies that $h$ is $\mu$-strongly convex for some $\mu> 0$, and $\nabla h$ is Lipschitz continuous with modulus $L\geq \mu$. Then, we have
    \begin{equation}
    \label{cond_1}
        \begin{aligned}
            \max_{\tau\in \gp_m} &\sum_{i=1}^m (L\sigma_{i}(X_1)+d^1_i)(\mu\sigma_{\tau(i)}(X_2)+d^2_{\tau(i)}) +\sum_{i=1}^m (\mu\sigma_{i}(X_1)+d^1_i)(L\sigma_{\tau(i)}(X_2)+d^2_{\tau(i)}) \\
            &-\sum_{i=1}^m (\mu \sigma_i(X_1)+d^1_i)(L\sigma_i(X_1)+d^1_i)-\sum_{i=1}^m (\mu \sigma_i(X_2)+d^2_i)(L\sigma_i(X_2)+d^2_i)\geq 0.
            \end{aligned}
    \end{equation}
\end{lemma}
\begin{proof}
    Define $\phi(\cdot):=h(\cdot)-\frac{\mu}{2}\|\cdot\|_F^2$, and let us rewrite \cref{base_ineq} as 
    \begin{equation}\label{hehehahaha}
0\le (L-\mu) \langle  \nabla \phi(X_1)-\nabla \phi (X_2), X_1-X_2\rangle - \|\nabla \phi(X_1)-\nabla \phi(X_2)\|_F^2.
    \end{equation}
    For the right hand side of \eqref{hehehahaha}, a direct computation shows that 
    \begin{equation}
        \label{eq:541}
        \begin{aligned}
            &(L-\mu) \langle  \nabla \phi(X_1)-\nabla \phi (X_2), X_1-X_2\rangle - \|\nabla \phi(X_1)-\nabla \phi(X_2)\|_F^2 \\ 
             &=\langle  \nabla \phi(X_1)-\nabla \phi (X_2), (L-\mu)X_1- \nabla \phi(X_1)-((L-\mu)X_2-\nabla \phi(X_2) )\rangle \\
            & = \langle -\nabla \phi(X_1),  (L-\mu)X_2- \nabla \phi(X_2)\rangle+\langle -\nabla \phi(X_2),  (L-\mu)X_1- \nabla \phi(X_1)\rangle \\
            &\quad-  \langle -\nabla \phi(X_1),  (L-\mu)X_1- \nabla \phi(X_1)\rangle-\langle -\nabla \phi(X_2),  (L-\mu)X_2- \nabla \phi(X_2)\rangle \\
            & =\langle \mu X_1 -\nabla h(X_1),  LX_2- \nabla h(X_2)\rangle+\langle \mu X_2-\nabla h(X_2),  LX_1- \nabla h(X_1)\rangle \\
            &\quad - \langle \mu X_1 -\nabla h(X_1),  LX_1- \nabla h(X_1)\rangle-\langle \mu X_2-\nabla h(X_2),  LX_2- \nabla h(X_2)\rangle\\
            & =: S_1 + S_2,
         \end{aligned}
    \end{equation}
    where $S_1:= \langle \mu X_1 -\nabla h(X_1),  LX_2- \nabla h(X_2)\rangle+\langle \mu X_2-\nabla h(X_2),  LX_1- \nabla h(X_1)\rangle$ and $S_2:= - \langle \mu X_1 -\nabla h(X_1),  LX_1- \nabla h(X_1)\rangle-\langle \mu X_2-\nabla h(X_2),  LX_2- \nabla h(X_2)\rangle$. 

    We now rewrite $S_1$ and $S_2$. We start by noting that for $S_2$, its two summands can be rewritten as follows using \cref{eq:X12G12}: for $i=1,2$,
    \begin{equation}
        \label{eq:xii}
        -\langle \mu X_i -\nabla h(X_i),  LX_i- \nabla h(X_i)\rangle=-\sum_{j=1}^m (\mu \sigma_j(X_i)+d^i_j)(L \sigma_j(X_i)+d^i_j).
    \end{equation}
    Next, for $S_1$, notice that
   \begin{equation}
    \label{eq:542}
    \begin{aligned}
        S_1 &= \langle \mu X_1 -\nabla h(X_1),  LX_2- \nabla h(X_2)\rangle+ \langle \mu X_2 -\nabla h(X_2),  LX_1- \nabla h(X_1)\rangle \\
        &\overset{\rm{(a)}}{=}\left\langle R_1\begin{bmatrix}
            \mu\Sigma_1+D_1 & 0
        \end{bmatrix}P_1^\top, R_2\begin{bmatrix}
            L\Sigma_2+D_2 & 0
        \end{bmatrix}P_2^\top  \right\rangle \\
        &\quad+ \left\langle R_1\begin{bmatrix}
            L\Sigma_1+D_1 & 0
        \end{bmatrix}P_1^\top, R_2\begin{bmatrix}
            \mu\Sigma_2+D_2 & 0 
        \end{bmatrix}P_2^\top  \right\rangle\\
        &=\left\langle R_2^\top R_1\begin{bmatrix}
            \mu\Sigma_1+D_1 & 0 
        \end{bmatrix}P_1^\top P_2, \begin{bmatrix}
            L\Sigma_2+D_2 & 0 
        \end{bmatrix}  \right\rangle \\
        &\quad +\left\langle R_2^\top R_1\begin{bmatrix}
            L\Sigma_1+D_1 & 0 
        \end{bmatrix}P_1^\top P_2, \begin{bmatrix}
            \mu\Sigma_2+D_2 & 0 
        \end{bmatrix}  \right\rangle,
       \end{aligned}
   \end{equation}
   where in (a) we have used \cref{eq:X12G12}. Using the above display and 
\cref{prop2-6}, we see that
   \begin{equation*}
     S_1 \le \max_{\tau\in \gp_m} \sum_{i=1}^m (\mu\sigma_{i}(X_1)+d^1_i)(L\sigma_{\tau(i)}(X_2)+d^2_{\tau(i)}) + \sum_{i=1}^m (L\sigma_{i}(X_1)+d^1_i)(\mu\sigma_{\tau(i)}(X_2)+d^2_{\tau(i)}).
   \end{equation*}
   The desired conclusion now follows immediately upon combining the above displays. 
\end{proof}
We already know that when $(\bar U,\bar V)$ is a stationary point of $F_r$ in \cref{prob_relaxation}, then $\bar X=\bar U\bar V^\top$ is a pseudo stationary point of $f$ in \cref{prob_nuclear}. The next step is to identify more structural information of $\bar X$ and $\nabla h(\bar X)$ when $(\bar U,\bar V)$ is a second-order stationary point of $F_r$. The first task is to rewrite the second-order optimality conditions such that it is aligned with the SVD of $\bar X$. To be more precise, the second-order optimality condition can be written as $\nabla^2F_r(\bar U,\bar V)[U,V]^2\geq 0$ for all $(U,V)\in \R^{m\times r}\times \R^{n\times r}$. To be more aligned with the SVDs of $(\bar U,\bar V)$ given in \cref{stationary_point}, we may rewrite it as 
\begin{equation}\label{soc_eq1-1-1}
\forall (U,V)\in \R^{m\times r}\times \R^{n\times r},~ ~~~  \nabla^2F_r(\bar U,\bar V)[(R  UQ^\top,P  VQ^\top),(R  UQ^\top,P  VQ^\top)]\geq 0,               \end{equation}
where $(R,Q)\in \cO_U$ and $(P,Q)\in \cO_V$.
Note that using~\cite[Equation (3.14)]{li2017geometry},\footnote{Note that \cite{li2017geometry} assumed $h$ is convex, but the derivation of Equation (3.14) there does not rely on the convexity assumption.} we have
\begin{equation}
    \label{soc_eq1}
    \begin{aligned}
                &\, \nabla^2F_r(\bar U,\bar V)[(R  UQ^\top,P  VQ^\top),(R  UQ^\top,P  VQ^\top)] \\
        =&\,  2\langle R^\top\nabla h(\bar X)P,  U  V^\top\rangle + \lambda (\|U\|_F^2+\|V\|^2_F)+\nabla^2h(\bar X)[R\left(R^\top \bar U Q   V^\top + U Q^\top\bar V^\top P\right)P^\top]^2.
    \end{aligned}
\end{equation}

Next, we introduce a natural partition for matrices of sizes $\R^{m\times r}$ and $\R^{n\times r}$. Let $(\bar U, \bar V)\in \R^{m\times r}\times \R^{n\times r}$ satisfy $\rank(\bar U) = \rank(\bar V)$ (which holds in particular when $(\bar U, \bar V)$ is a stationary point of $F_r$, thanks to \cref{stationary_point}). Denote this common rank by $s = \rank(\bar U) = \rank(\bar V) \le r$. We can then partition any matrices $U\in \R^{m\times r}$ and $V\in\R^{n\times r}$ into the following block form:
\begin{equation}
    \label{blockuv}
        U=\begin{bmatrix}
            U_{11} & U_{12} \\
            U_{21} & U_{22}
        \end{bmatrix}\quad\text{and}\quad V=\begin{bmatrix}
            V_{11} & V_{12} \\
            V_{21} & V_{22}
        \end{bmatrix},
\end{equation}
where $U_{11},V_{11}\in \R^{s\times s}$, $U_{12},V_{12}\in \R^{s\times (r-s)}$, $U_{21}\in \R^{(m-s)\times s}$, $V_{21}\in \R^{(n-s)\times s}$, $U_{22}\in \R^{(m-s)\times (r-s)}$, $V_{22}\in \R^{(n-s)\times (r-s)}$. Note that when $\bar U$ and $\bar V$ are of full rank, \ie, $s = r$, the matrices $U_{12}$, $U_{22}$, $V_{12}$ and $V_{22}$ are null. Now, we are ready to state the following characterization of second-order stationarity of $F_r$ in \cref{prob_relaxation}.
\begin{prop}[Second-order stationarity]
    \label{prop2-1}
    A pair $(\bar U, \bar V) \in \R^{m\times r}\times \R^{n\times r}$ is a second-order stationary point of $F_r$ in \eqref{prob_relaxation} if and only if both of the following two conditions hold:
    \begin{enumerate}[label=(\roman*)]
        \item\label{prop2-1-i} There exist $R\in \cO^m$, $P\in \cO^n$ and $Q\in\cO^r$ such that $(R, Q)\in \cO_{\bar U}$, $ (P, Q)\in \cO_{\bar V} $, $\sigma (\bar U) = \sigma (\bar V)$ and $ \nabla h( \bar U \bar V^\top ) = - R \,\wDiag(d)\, P^\top $ for some $d\in \R^m$ satisfying $d_1 = \cdots = d_s = \lambda$ and $ d_{s+1} \ge \cdots \ge d_m \ge 0 $, where $s = \rank (\bar U) = \rank(\bar V)$.
        
        \item\label{prop2-1-ii} For any $U\in \R^{m\times r}$ and $V\in\R^{n\times r}$, it holds that\footnote{Here, we use the partition \cref{blockuv} with respect to $(\bar U,\bar V)$; this is well defined because $\sigma (\bar U) = \sigma (\bar V)$ holds in \cref{prop2-1-i}.} 
        \begin{equation}
            \label{newsoF}
            \begin{aligned}
                &-2\lambda\tr(U_{11}^\top V_{11}+U_{12}V_{12}^\top)-2\tr(D^\top(U_{21}V_{21}^\top+U_{22}V_{22}^\top))\\
                & \quad +\lambda (\|U\|_F^2+\|V\|^2_F) +\nabla^2h(\bar U \bar V^\top )\left[R\begin{bmatrix}
                    U_{11}\Sigma + \Sigma V_{11}^\top & \Sigma V_{21}^\top \\
                    U_{21}\Sigma  & 0 
                \end{bmatrix}P^\top\right]^2\geq 0 ,
            \end{aligned} 
        \end{equation}
        where $\Sigma = \Diag(\sigma_1(\bar U) , \dots, \sigma_s(\bar U))\in\R^{s\times s}$, and $D = \wDiag(d_{s+1},\dots, d_m) \in\R^{(m-s)\times (n-s)}$ with $d_i$ given in \cref{prop2-1-i}.
    \end{enumerate} 
    Moreover, if $s=\rank(\bar U)<r$, then \cref{prop2-1-i} and \cref{prop2-1-ii} imply that\footnote{We would like to point out this last claim (i.e., ``if $s=\rank(\bar U)<r$, then \cref{prop2-1-i} and \cref{prop2-1-ii} imply that $\|\nabla h(\bar U \bar V^\top)\|_2\leq \lambda$.") has been proved in \cite[Theorem 1]{boumal2024}. We include its proof for completeness.} $\|\nabla h(\bar U \bar V^\top)\|_2\leq \lambda$.
\end{prop}
\begin{proof}
    A pair $(\bar U, \bar V) \in \R^{m\times r}\times \R^{n\times r}$ is a second-order stationary point of $F_r$ if and only if it satisfies both the first- and second-order optimality conditions. By \cref{stationary_point}, the first-order optimality condition is equivalent to \cref{prop2-1-i}. 
    
    Using \cref{soc_eq1}, the second-order condition in \cref{soc_eq1-1-1} can be rewritten as that for all $(U,V)\in \R^{m\times r}\times \R^{n\times r}$,
    \begin{align}
       \notag  &2\langle R^\top\nabla h(\bar X)P,  U  V^\top\rangle + \lambda (\|U\|_F^2+\|V\|^2_F) +\nabla^2h(\bar X)[\bar U Q   V^\top P^\top+ R  U Q^\top\bar V^\top]^2\geq 0  \\
      \notag \overset{\rm (a)}{\iff} &-2\left\langle \begin{bmatrix}
         \lambda I_s & 0 \\
         0 & D
       \end{bmatrix},  U  V^\top\right\rangle + \lambda (\|U\|_F^2+\|V\|^2_F) \\
       &+\nabla^2h(\bar X)\left[R\left(\begin{bmatrix}
          \Sigma & 0 \\
           0 & 0
       \end{bmatrix}   V^\top+    U\begin{bmatrix}
          \Sigma & 0 \\
           0 & 0
       \end{bmatrix}\right)P^\top\right]^2\geq 0 \\
       \notag \overset{\rm (b)}{\iff} &-2\lambda\tr(U_{11} V_{11}^\top+U_{12}V_{12}^\top)-2\tr(D^\top(U_{21}V_{21}^\top+U_{22}V_{22}^\top))\\
            & + \lambda (\|U\|_F^2+\|V\|^2_F) +\nabla^2h(\bar X)\left[R\begin{bmatrix}
                U_{11}\Sigma+\Sigma V_{11}^\top & \Sigma V_{21}^\top \\
                U_{21}\Sigma  & 0 
            \end{bmatrix}P^\top\right]^2\geq 0,
    \end{align}
    where (a) follows from the decomposition for $-\nabla h(\bar X)$ in  \cref{prop2-1-i} and the definition of $D$ and $\Sigma$, and (b) from the definition of the blocks in \cref{blockuv}. This shows that \cref{prop2-1-i} and \cref{prop2-1-ii} together form an equivalent characterization of the second-order stationary points. 
 
    We next prove the second claim. Assume that $s=\rank(\bar U)<r$. Note that this implies that $U_{22}$ and $V_{22}$ are not null. Hence, we can take $U$ and $V$ in \cref{newsoF} to be the matrices with the blocks $U_{22}= [e_1~0]$ and $V_{22}=[e_1~0]$ and $U_{11}$, $U_{12}$, $U_{21}$, $V_{11}$, $V_{12}$, $V_{21}$ all being zero matrices to deduce that
    \[      -2\sigma_1( D ) + 2\lambda\geq  0.\]
    The desired conclusion now follows immediately from the above display and the decomposition of $\nabla h(\bar U \bar V^\top)$ in \cref{prop2-1-i}.
 \end{proof}

Using the second-order condition in \cref{prop2-1}, we now derive the information on the singular values of $\nabla h(\bar X)$, when $\bar X=\bar U\bar V^\top$ with $(\bar U,\bar V)$ being a second-order stationary point of $F_r$ in \cref{prob_relaxation}. The technique used here is not new, and similar arguments can be found in \cite{ha2020equivalence} and \cite{boumal2024}.
\begin{thm}[Singular values for 2nd stationary points]
    \label{prop4-7}
    Let $(\bar U,\bar V) \in \R^{m\times r}\times \R^{n\times r}$ be a second-order stationary point of $F_r$ in \cref{prob_relaxation}, $s = \rank(\bar U)$ and $d \in\R^m$ be given in Proposition~\ref{prop2-1}. Then,\footnote{We define $d_{m+1}= 0$.} $d_{s+1} \le \lambda + \tilde L \sigma_r (\bar U \bar V^\top)$, where
    \[
    \tilde L:=\sup_{
    \begin{subarray}{c}
    Y\in \R^{m\times n},\, \|Y\|_F=1,\\
    \rank(Y)=2
    \end{subarray}
    } \nabla^2h(\bar U \bar V^\top)[Y,Y].
    \]
\end{thm}
\begin{proof}
    We first consider the case when $\rank(\bar U) < r$. Since $\rank(\bar U) < r$, we have $\sigma_r (\bar U \bar V^\top) = 0$. Therefore,
    \[
        d_{s+1} \le \|d\|_\infty = \|\nabla h(\bar U \bar V^\top)\|_2 \le \lambda = \lambda + \tilde L \sigma_r (\bar U \bar V^\top),
    \]
    where the first equality follows from \cref{haha}(i) and the second inequality follows from \cref{prop2-1}. 

    We next consider the case where $\rank(\bar U) = r$. If $r=m$, then $d_{m+1} = 0$
    and the desired conclusions then hold trivially. Thus, from now on, we assume $r < m$.
    
    By \cref{stationary_point}, $\rank(\bar V) = \rank(\bar U) = r$. Therefore, in this case, the blocks $U_{12}, U_{22}, V_{12}, V_{22}$ in \eqref{blockuv} are null. 
    Since $(\bar U,\bar V) $ is a second-order stationary point of $F_r$ in \cref{prob_relaxation}, by \cref{prop2-1}, it satisfies the following inequality for any matrices $U\in \R^{m\times r}$ and $V\in \R^{n\times r}$:
    \begin{equation}
        \label{ineq:proof-1}
        \begin{aligned}
        0   \le & -2\lambda \tr(U_{11}^\top V_{11})-2\tr(D^\top U_{21}V_{21}^\top)+\lambda(\|U\|_F^2 + \|V\|_F^2)\\
        &+\nabla^2 h(\bar U \bar V^\top)\left[ R\begin{bmatrix}
                U_{11}\Sigma + \Sigma V_{11}^\top & \Sigma V_{21}^\top \\
                 U_{21}\Sigma & 0_{(m-r)\times (n-r)}
            \end{bmatrix} P^\top   \right]^2.
        \end{aligned}
    \end{equation}
    Taking $U$ and $V$ to be the matrices with $U_{11}$ and $V_{11}$ being zero and $U_{ 21 } = [e_r\ 0]^\top = e_1 e_r^\top\in \R^{(m-r)\times r}$ and $V_{ 21 } = [e_r\  0 ]^\top = e_1e_r^\top  \in \R^{(n-r)\times r}$, we have that $\Sigma V_{21}^\top = \sigma_r (\bar U) e_r e_1^\top $ and $U_{21} \Sigma = \sigma_r (\bar U) e_1 e_r^\top$ and that $\tr(D^\top U_{21}V_{21}^\top) = e_1^\top D e_1 = d_{r+1}$. Substituting these into~\eqref{ineq:proof-1} yields
    \begin{equation*}
        \begin{split}
            0 & \le -2d_{r+1} + 2\lambda + \nabla^2 h(\bar U \bar V^\top)\left[ \sigma_r (\bar U) R 
            \begin{bmatrix}
                 0_{ r\times r}& e_r e_1^\top \\
                 e_1e_r^\top &  0_{(m-r)\times (n-r)}
            \end{bmatrix}P^\top    \right]^2 \\
            & \le -2d_{r+1} + 2\lambda + 2 \sigma_r^2 (\bar U) \tilde L = -2d_{r+1} + 2\lambda + 2 \sigma_r (\bar U \bar V^\top) \tilde L,
        \end{split}
    \end{equation*}
    where the second inequality follows from the fact that $\left\| \sigma_r (\bar U) R \begin{bmatrix}
                 0_{r\times r} & e_r e_1^\top \\
                 e_1e_r^\top &  0_{(m-r)\times (n-r)}
            \end{bmatrix} P^\top  \right\|_F = \sqrt{2} \sigma_r (\bar U)$ and the definition of $\tilde L$, and the equality follows from \cref{haha}(ii). Hence, $d_{r+1} \le \lambda + \tilde L\, \sigma_r ( \bar U \bar V^\top )$.
\end{proof}
 In view of \cref{prop4-7}, the definition of $d$ in \cref{prop2-1} and \cref{subd_nuclear}, for a second-order stationary point $(\bar U,\bar V)$ of $F_r$ in \cref{prob_relaxation}, by setting $\bar X = \bar U\bar V^\top$, we see that if $\sigma_r(\bar X)=0$, then $\bar X$ is a stationary point of $f$, as proved in \cite{boumal2024}; see, also \cite[Section 3]{zhang2023preconditioned}, and \cite{journee2010low,boumal2016non,boumal2020deterministic} for similar results. This is formally presented as the next corollary.
\begin{coro}
\label{rand_defi_desirable}
    For any second-order stationary point $(\bar U, \bar V)\in \R^{m\times r}\times \R^{n\times r}$ of $F_r$ in \eqref{prob_relaxation} satisfying $\rank(\bar U)<r$, $ \bar U \bar V^\top  $ is a stationary point of $f$ in \eqref{prob_nuclear}. 
\end{coro}

Intuitively, based on the bound on $d_{s+1}$ in \cref{prop4-7}, we can say that the smaller $\sigma_r(\bar X)$ is, the closer $\bar X$ is to being a stationary point of $f$ in \cref{prob_nuclear}.

\section{Characterization of $r$-factorizability for strongly convex functions}
\label{boundr}
Suppose we have an $h\in \fS(L,\mu,r^*)$ in hand such that the corresponding $f$ in \cref{prob_nuclear} is not $r$-factorizable. By definition, we know there exists $X^*\in \R^{m\times n}$ such that $X^*$ is a stationary point of $f$. Let $(\bar U,\bar V)$ be the second-order stationary point of $F_r$ in \cref{prob_nuclear} such that $\bar X=\bar U\bar V^\top$ is not a global minimizer of $f$ (or equivalently, not a stationary point of $f$ since $f$ is strongly convex thanks to $h\in \fS(L,\mu, r^*)$). Then according to \cref{stationary_point}, we know $\bar X$ is a pseudo stationary point of $f$. In this case, we have the bound \cref{cond_1} on the singular values of $X^*$, $\nabla h(X^*)$, $\bar X$ and $\nabla h(\bar X)$. In addition, for a pseudo stationary point $X$ of $f$, it is a stationary point of $f$ if and only if $\|\nabla h(X)\|_2\leq \lambda$ in view of \cref{defn-psp} and \cref{subd_nuclear}. Finally, the singular value of $\nabla h(\bar X)$ is further constrained by \cref{prop4-7}. Putting all these conditions together, we have the following result.
\begin{prop}
\label{propb1}
   Let $L\geq \mu >0$, $r\in [m]$ and $r^*\in [m]\cup \{0\}$. Assume that there exists an $h\in \fS(L,\mu,r^*)$ (see \cref{defn_class}) such that $f$ in \cref{prob_nuclear} is not $r$-factorizable.
Then, there exist $x,g,y,v\in \R^m$ with $\|g\|_\infty>\lambda$ and $\tau\in \gp_m$ such that 
        \begin{subequations}
        \begin{equation}
        \label{eqb11}
            \begin{aligned}
                  && &\sum_{i=1}^m (L x_i+g_i)(\mu y_{\tau(i)}+v_{\tau(i)})+\sum_{i=1}^m (\mu x_i+g_i)(Ly_{\tau(i)}+v_{\tau(i)}) \\
            &&&-\sum_{i=1}^m(Lx_i+g_i)(\mu x_i+g_i)-\sum_{i=1}^m(Ly_i+v_i)(\mu y_i+v_i)\geq 0, 
            \end{aligned}
        \end{equation}
        and
        \begin{align}
            &\forall i\in [r],~x_i>0,~g_i=\lambda,\quad \forall i\in [r^*],~y_i>0,~v_i=\lambda,\label{eqb12}\\
            &\forall i\in [m]\setminus [r],~x_i=0,~g_i\in [0,\lambda+L\min_{j\in [r]}x_j],\quad \forall i\in [m]\setminus [r^*], ~y_i=0,~v_i\in [0,\lambda].  \label{eqb13}
        \end{align}
        \end{subequations}
\end{prop}
\begin{proof}
     By assumption, we can select $h\in\fS(L,\mu,r^*)$ and $X_2\in \R^{m\times n}$ with $\rank(X_2)=r^*$, $-\nabla h(X_2)\in \lambda\partial \|X_2\|_*$, and $f$ is not $r$-factorizable. The latter means we can find $(\bar U,\bar V)$ being a second-order stationary point of $F_r$ in \cref{prob_relaxation} and $X_1=\bar U\bar V^\top$ is not a stationary point of $f$. 
    
    Applying \cref{stationary_point} (see also \cref{haha}) to $(\bar U, \bar V)$ and using  \cref{subd_nuclear} and the condition that $-\nabla h(X_2)\in \lambda\partial \|X_2\|_*$, we can write
    \begin{equation}
        \label{eq:X12G12ce}
        \begin{aligned}
            &X_1=R_1\begin{bmatrix}
                \Sigma_1 & 0 
            \end{bmatrix}   P_1^\top,~ -\nabla h(X_1)=R_1\begin{bmatrix}
                D_1 & 0 
            \end{bmatrix}P_1^\top, \\
            &X_2=R_2\begin{bmatrix}
                \Sigma_2 & 0 
            \end{bmatrix}P_2^\top,~ -\nabla h(X_2)=R_2\begin{bmatrix}
                D_2 & 0 
            \end{bmatrix}P_2^\top, 
        \end{aligned} 
    \end{equation}
    for some $R_i\in \cO^m$ and $P_i\in \cO^n$ and $m\times m$ diagonal matrices $\Sigma_i$ and $D_i$, $i=1,2$,
    where $\diag(\Sigma_i)\in \R^{m}_+$ consisting of all the singular values of $X_i$ in descending order, $d^i:= \diag(D_i)\in \R^m_+$ with $d^i_1=\dots=d^i_{\rank(X_i)}=\lambda$, for $i=1,2$. Clearly, we have $\rank(X_1)=r$, otherwise by \cref{rand_defi_desirable} we can conclude that $X_1$ is a stationary point of $f$, leading to a contradiction. Applying \cref{prop5-4}, there exists $\bar\tau\in\gp_m$ such that  
    \begin{equation}
    \label{obj_sc_ce}
            \begin{aligned}
        &\sum_{i=1}^m (L\sigma_{i}(X_1)+d^1_i)(\mu\sigma_{\bar\tau(i)}(X_2)+d^2_{\bar\tau(i)}) +\sum_{i=1}^m (\mu\sigma_{i}(X_1)+d^1_i)(L\sigma_{\bar\tau(i)}(X_2)+d^2_{\bar\tau(i)}) \\
            &-\sum_{i=1}^m (L\sigma_i(X_1)+d^1_i)(\mu \sigma_i(X_1)+d^1_i)-\sum_{i=1}^m (L\sigma_i(X_2)+d^2_i)(\mu \sigma_i(X_2)+d^2_i)\geq 0,
    \end{aligned}
    \end{equation}
    where $L$ and $\mu$ are defined in \cref{defn_class} for the $h$ we selected.
    Next, applying \cref{prop4-7}, we know for all $i\geq r+1$, it holds that $d_i^1\leq \lambda+L\sigma_r(X_1)$; in addition, it must hold that $\|d^1\|_\infty > \lambda$ for otherwise, \cref{eq:X12G12ce} and \cref{subd_nuclear} would imply that $X_1$ is a stationary point of $f$, which is a contradiction. On the other hand, using the fact that $-\nabla h(X_2)\in \lambda\partial \|X_2\|_*$, \cref{eq:X12G12ce} and \cref{subd_nuclear}, we know $d_i^2\leq \lambda$ for all $i\in [m]$. This means that $(x,g,y,v,\tau) =(\diag(\Sigma_1),\diag(D_1),\diag(\Sigma_2),\diag(D_2),\bar\tau)$ satisfies \eqref{eqb11}--\cref{eqb13} and $\|g\|_\infty > \lambda$.  
\end{proof}
To provide a complete characterization, it is important to know whether the converse of \cref{propb1} holds. The first issue we need to solve is that, a feasible pair $(x,g,y,v)$ satisfying \cref{eqb12} and \cref{eqb13} does not necessarily satisfy that $x$ and $y$ are sorted in descending order, which means that they cannot be the singular values of any matrix. \revise{We will also need $(g_i)_{i=r+1}^m$ and $(v_i)_{i=r^*+1}^m$ to be ordered so that the diagonal decompositions used below match the ordering convention in \cref{prop2-1}.} Nevertheless, we can reduce to \revise{this ordered case} by doing an additional permutation.
\begin{lemma}
\label{lemma5-2}
    Let $L\geq \mu >0$, $r\in [m]$ and $r^*\in [m]\cup \{0\}$. Assume that $(x,g,y,v,\tau)$ satisfies \cref{eqb11}--\cref{eqb13} with $\|g\|_\infty>\lambda$. Then, there exist $\tau_1,\tau_2\in \gp_m$ such that\footnote{Here for a vector $x\in \R^n$ and a permutation $\tau$, $\tau(x)\in \R^n$ is defined as $\tau(x)_i=x_{\tau(i)}$ for all $i\in [n]$.}
    \begin{equation}
        \label{eqb2}
        \begin{aligned}
        &\revise{\bar x:=\tau_1(x),~\bar y:=\tau_2(y)} \text{ are sorted in descending order, } \revise{\bar g :=\tau_1(g)\ {and}\ \bar v :=\tau_2(v)}\\
        &\revise{\text{have the subvectors }(\bar g_i)_{i=r+1}^m\text{ and }(\bar v_i)_{i=r^*+1}^m\text{ sorted in descending order,}}
        \end{aligned}
    \end{equation}
  and $(\bar x,\bar g,\bar y,\bar v,\rho)$ also satisfies \cref{eqb11}--\cref{eqb13} with $\|\bar g\|_\infty>\lambda$, where $\rho:=\tau_2^{-1}\tau\tau_1$.
\end{lemma}
\begin{proof}
 \revise{The permutations $\tau_1$ and $\tau_2$ satisfying \cref{eqb2} exist because the constraints on $x,y$ in \cref{eqb12} and \cref{eqb13} imply that $x_i>0$ if and only if $i\in [r]$, and $y_j>0$ if and only if $j\in [r^*]$. Hence, after sorting the positive parts of $x$ and $y$, we still have the freedom to permute the zero blocks of $x$ and $y$, which can be used to sort the remaining parts of $g$ and $v$. For such permutations, we next verify \cref{eqb11}--\cref{eqb13}.} According to the definition of $\tau_1$ and $\tau_2$, we have 
    \begin{equation}
        \label{eqb4}
        \tau_1([r])=[r],\quad \tau_1([m]\setminus [r])=[m]\setminus[r],\quad \tau_2([r^*])=[r^*],\quad \tau_2([m]\setminus [r^*])=[m]\setminus[r^*] .
    \end{equation}
Therefore, $(\bar x,\bar g,\bar y,\bar v)$ satisfies \cref{eqb12} and \cref{eqb13}. To verify \cref{eqb11}, notice that the four sums in \cref{eqb11} takes the form $\sum_{i=1}^m a_ib_{\tau(i)}$ for some $a,b\in \R^m$. We also note that for any $a,b,c,d\in \R^m$ with $a=\tau_1(c)$ and $b=\tau_2(d)$, it holds that
\begin{align}
    \label{permutation_eq}
    \sum_{i=1}^m a_ib_{\rho(i)}\overset{\rm(a)}{=}\sum_{i=1}^m c_{\tau_1(i)}d_{\tau_2\rho(i)} \overset{\rm (b)}{=} \sum_{i=1}^mc_{\tau_1(i)}d_{\tau\tau_1(i)}\overset{\rm (c)}{=}\sum_{i=1}^mc_id_{\tau(i)},  
\end{align}
    where (a) holds because $a=\tau_1(c)$ and $b=\tau_2(d)$,  (b) holds as $\rho=\tau_2^{-1}\tau\tau_1$, and in (c) we have used the substitution $i\gets \tau_1^{-1}(i)$. Consequently, to verify \cref{eqb11} for $(\bar x,\bar g,\bar y,\bar v,\rho)$, it suffices to use \cref{eqb2}, \cref{permutation_eq} and \cref{eqb11} for $(x,g,y,v,\tau)$. Finally, the condition $\|\bar g\|_\infty>\lambda$ is clear since $\bar g$ is a permutation of $g$ and $\|g\|_\infty>\lambda$.
\end{proof}
Consequently, if there exists $(x,y,g,v,\tau)$ satisfying \cref{eqb11}--\cref{eqb13} with $\|g\|_\infty>\lambda$, then there is no loss of generality if we assume $x$ and $y$ are sorted in descending order \revise{and the subvectors $(g_i)_{i=r+1}^m$ and $(v_i)_{i=r^*+1}^m$ are also sorted in descending order}. In this case, we would like to construct $\bar X$, $\nabla h(\bar X)$, $X^*$ and $\nabla h(X^*)$ such that $\sigma(\bar X)=x$, $\sigma(X^*)=y$ and the bound in \cref{base_ineq} is satisfied. To this end, we first define
\begin{equation}
        \label{eq:ce1}
        \begin{aligned}
        \bar X:=\wDiag(x),~X^*:=\wDiag(\tau(y)),~\bar G:= 
            \wDiag(g)  
        ,~G^*:= 
            \wDiag(\tau(v)),
        \end{aligned}
\end{equation}
In this case, by direct calculation, we have 
\begin{equation}
\label{ineq_construct}
   \begin{aligned}
      & (L-\mu) \langle (G^*+\mu X^*)-(\bar{G}+\mu \bar{X}), \bar{X}-X^*\rangle- \|(\bar{G}+\mu \bar{X})-(G^*+\mu X^*)\|_F^2 \\
    &=  \langle (G^*+\mu X^*)-(\bar{G}+\mu \bar{X}), L\bar{X}+\bar G-(LX^*+G^*)\rangle \\
    &=\langle G^*+\mu X^*,L\bar X+\bar G \rangle+\langle \bar G+\mu \bar X,LX^*+ G^* \rangle\\
    &~~~~~-\langle G^*+\mu X^*,LX^*+G^* \rangle-\langle \bar G+\mu \bar X,L\bar X+ \bar G \rangle \\
    &\overset{\rm (a)}{=}\sum_{i=1}^m (L x_i+g_i)(\mu y_{\tau(i)}+v_{\tau(i)})+\sum_{i=1}^m (\mu x_i+g_i)(Ly_{\tau(i)}+v_{\tau(i)}) \\
    &~~~~-\sum_{i=1}^m(Lx_i+g_i)(\mu x_i+g_i)-\sum_{i=1}^m(Ly_i+v_i)(\mu y_i+v_i)\overset{\rm (b)}{\geq} 0
   \end{aligned}
\end{equation}
where (a) follows from the diagonal structures in \cref{eq:ce1}, and (b) follows from \cref{eqb11}. The display \cref{ineq_construct} implies that letting $\nabla h(\bar X)=-\bar G$ and $\nabla h(X^*)=-G^*$ does not violate \cref{base_ineq} when $h$ is $\mu$-strongly convex and $\nabla h$ is $L$-Lipschitz. The next question is, does there exist such a function $h\in C^2(\R^{m\times n})$? Notice that $-\nabla h(X^*)\in \lambda\partial\|X^*\|_*$ if such a function exists, which means $X^*$ is a stationary point of $f$ in \cref{prob_nuclear}, and $h\in \fS(L,\mu, r^*)$. This question is a special case of the extension of strongly convex function \cite{azagra2017whitney}. However, for our purpose, we need to further ensure that there exists a second-order stationary point $(\bar U,\bar V)$ of the corresponding $F_r$ in \cref{prob_relaxation} such that $\bar U\bar V^\top =\bar X$, which will be achieved by adding quadratic penalty on the off diagonal entries due to the structure of the second-order optimality condition in \cref{sec:stationary_points}. Let us now provide the construction of such example.
\begin{lemma}
\label{prop5-6}
      Let $L\geq \mu >0$, $r\in [m]$ and $r^*\in [m]\cup \{0\}$. Assume that $(x,g,y,v,\tau)$ satisfies \cref{eqb11}--\cref{eqb13}, $x$ and $y$ are sorted in descending order, \revise{and the subvectors $(g_i)_{i=r+1}^m$ and $(v_i)_{i=r^*+1}^m$ are also sorted in descending order}.   
     Let $\bar X$, $X^*$, $\bar G$, $G^*\in \R^{m\times n}$ be defined as \cref{eq:ce1}.
    If $G^*+\mu X^*\neq \bar G+\mu \bar X$, we define a quadratic function $h$ as follows:
\begin{equation}
        \label{defnh1}
        \begin{aligned}
            h(X)=&\frac L2\sum_{i=1}^m\sum_{j\neq i}^n X_{ij}^2+ \frac \mu2 \sum_{i=1}^m (X_{ii}-( \bar{X})_{ii})^2-\langle  \bar G,X\rangle\\
        &+\frac{(\langle X- \bar{X}, -G^*-\mu X^*+ \bar G+\mu  \bar{X} \rangle)^2}{2\langle X^*- \bar{X}, -G^*-\mu X^*+ \bar G+\mu  \bar{X}\rangle},  
        \end{aligned}
    \end{equation}
    Otherwise, we set
    \begin{equation}
        \label{defnh2}
        \begin{aligned}
            h(X)=&\frac L2\sum_{i=1}^m\sum_{j\neq i}^n X_{ij}^2+ \frac \mu2 \sum_{i=1}^m (X_{ii}-( \bar{X})_{ii})^2-\langle  \bar G,X\rangle.
        \end{aligned}
    \end{equation}
    Then $h$ is well defined and $\mu$-strongly convex, $\nabla h$ is Lipschitz continuous with modulus $L$, $\nabla h( \bar{X})=- \bar G$, and $\nabla h(X^*)=-G^*$. Moreover, if we define $F_r$ as in \cref{prob_relaxation} with the above $h$ and define $(\bar U,\bar V)\in \R^{m\times r}\times \R^{n\times r}$ as
        \begin{equation}
        \label{defn_u1v1}
        \bar U:= 
            \wDiag(\sqrt{\sigma_1( \bar{X})},\dots,\sqrt{\sigma_r( \bar{X})})  
\ \ {and}\ \ \bar V:= 
              \wDiag(\sqrt{\sigma_1( \bar{X})},\dots,\sqrt{\sigma_r( \bar{X})}),
    \end{equation}
    then $(\bar U, \bar V)$ is a second-order stationary point of $F_r$. In particular, if $\|g\|_\infty>\lambda$, then $f$ in \cref{prob_nuclear} is not $r$-factorizable given such an $h$.
\end{lemma}
\begin{proof}
    We first consider the case where $ \bar G+\mu  \bar{X}=G^*+\mu X^*$. In this case, the function $h$ in \cref{defnh2} is clearly well defined, and one can verify that $h$ is $\mu$-strongly convex, and $\nabla h$ is Lipschitz continuous with modulus $L$. Moreover, $\nabla h( \bar{X})=- \bar G$ and 
    \[
    \nabla h(X^*)=\mu (X^* -  \bar{X}) -  \bar G = -G^*.
    \]
    Now it remains to show that $(\bar U, \bar V)$ is a second-order stationary point of $F_r$. 
    
    \revise{We first observe that, by \cref{eqb12}, \cref{eqb13}, and the ordering assumption on $(g_i)_{i=r+1}^m$ and $(v_i)_{i=r^*+1}^m$ satisfies $g_1=\cdots=g_r=\lambda$ and, when $r<m$, $g_{r+1}\geq\cdots\geq g_m\geq 0$. Together with the definitions of $\bar G,\bar U,\bar V$ and the fact that $\nabla h(\bar X)=-\bar G$, this shows that condition \cref{prop2-1-i} in \cref{prop2-1} holds with $R=I_m$, $P=I_n$, $Q=I_r$, and $d=g$. Consequently, by \cref{prop2-1}, it remains to verify the second-order condition. Since $\rank(\bar U)=r$, the blocks $U_{12},V_{12},U_{22},V_{22}$ are void in the partition of \cref{blockuv}; hence it suffices to show that for all $U_{11},V_{11}\in \R^{r\times r},~U_{21}\in \R^{(m-r)\times r},~V_{21}\in \R^{(n-r)\times r}$, it holds that}
    \begin{equation}  \label{eq:ce3}
        \begin{aligned}
            &-2\lambda\tr(U_{11}^\top V_{11})-2\tr(D^\top U_{21}V_{21}^\top)
            +\lambda(\|U_{11}\|_F^2+\|V_{11}\|_F^2+\|U_{21}\|_F^2+\|V_{21}\|_F^2)\\
            &+\nabla^2h( \bar{X})\left[ \begin{bmatrix}
                U_{11} \Sigma+\Sigma V_{11}^\top & \Sigma V_{21}^\top \\
                U_{21}\Sigma & 0 
            \end{bmatrix} \right]^2\geq 0,
        \end{aligned}
    \end{equation}  
    where 
    \begin{equation}
        \label{defnbarw}
\begin{aligned}&
\Sigma=\diag(\sqrt{\sigma_1( \bar{X})},\dots,\sqrt{\sigma_r( \bar{X})})
\overset{(a)}=\diag(\sqrt{x_1},\dots,\sqrt{x_r})\in\R^{r\times r},\\
&D= 
            \wDiag(g_{r+1},\dots,g_m)  \in \R^{(m-r)\times (n-r)},
        \end{aligned}
    \end{equation}
    and in (a) we have used the fact that $x$ is a nonnegative vector sorted in descending order. Next, we will verify \eqref{eq:ce3}. To this end, we first use the representation of $h$ in \cref{defnh2} to deduce that
    \begin{equation}
        \label{eq:ce4}
        \begin{split}
        \nabla^2h( \bar{X})\left[ \begin{bmatrix}
            U_{11} \Sigma+\Sigma V_{11}^\top & \Sigma V_{21}^\top \\
            U_{21}\Sigma & 0  
        \end{bmatrix} \right]^2&\geq L\|\Sigma V_{21}^\top\|_F^2+ L\|U_{21}\Sigma\|_F^2\\
        &\overset{\rm{(a)}}{\geq} Lx_r(\|U_{21}\|^2_F+\|V_{21}\|_F^2),
        \end{split}
    \end{equation}
     where in (a) we have used the fact that $\|AB\|_F\geq \sigma_{\min}(A)\|B\|_F$ for any square matrix $A$. Therefore, it holds that 
     \begin{align*}
        &-2\lambda\tr(U_{11}^\top V_{11})-2\revise{\tr(D^\top U_{21}V_{21}^\top)}+ \lambda(\|U_{11}\|_F^2+\|V_{11}\|_F^2+\|U_{21}\|_F^2+\|V_{21}\|_F^2)\\
        &+\nabla^2h( \bar{X})\left[\begin{bmatrix}
                U_{11} \Sigma+\Sigma V_{11}^\top & \Sigma V_{21}^\top \\
                U_{21}\Sigma & 0
            \end{bmatrix}\right]^2 \\
            &\overset{\rm{(a)}}{\geq }-2\tr(D^\top U_{21}V_{21}^\top)+\lambda(\|U_{21}\|_F^2+\|V_{21}\|^2_F)+\nabla^2h( \bar{X})\left[\begin{bmatrix}
                U_{11} \Sigma+\Sigma V_{11}^\top & \Sigma V_{21}^\top \\
                U_{21}\Sigma & 0
            \end{bmatrix}\right]^2 \\
            &\overset{\rm{(b)}}{\geq }-2\|D^\top\|_{2} \|U_{21}\|_F\|V_{21}^\top\|_F+\lambda(\|U_{21}\|_F^2+\|V_{21}\|^2_F)+\nabla^2h( \bar{X})\left[\begin{bmatrix}
                U_{11} \Sigma+\Sigma V_{11}^\top & \Sigma V_{21}^\top \\
                U_{21}\Sigma & 0
            \end{bmatrix}\right]^2 \\
            &\overset{\rm{(c)}}{\geq }-2(\lambda+Lx_r)\|U_{21}\|_F\|V_{21}\|_F+\lambda(\|U_{21}\|_F^2+\|V_{21}\|^2_F)+Lx_r(\|U_{21}\|^2_F+\|V_{21}\|_F^2) \\
            &\overset{\rm{(d)}}{\geq}-(\lambda+Lx_r)(\|U_{21}\|_F^2+\|V_{21}\|^2_F)+\lambda(\|U_{21}\|_F^2+\|V_{21}\|^2_F)
            +Lx_r(\|U_{21}\|^2_F+\|V_{21}\|_F^2)=0,
     \end{align*}
    where in (a) we have used the Cauchy-Schwartz inequality to show that $\tr(U_{11}^\top V_{11})\leq\frac{1}{2}(\|U_{11}\|^2_F+\|V_{11}\|_F^2)$, in (b) we have used the fact $\tr(ABC) = \tr(CAB) \le \|C\|_F\|AB\|_F\le \|A\|_2\|C\|_F\|B\|_F$, in (c) we have used \cref{eqb13}, \cref{eq:ce4} and the definition of $D$ in \cref{defnbarw}, and in (d) we have used the fact $\|A\|_F\|B\|_F\leq \frac{1}{2}(\|A\|_F^2+\|B\|_F^2)$. This verifies \eqref{eq:ce3} and hence $(\bar U, \bar V)$ is a second-order stationary point of $F_r$.
    
    Next, we consider the case where $ \bar G+\mu  \bar{X}\neq G^*+\mu X^*$. By \cref{ineq_construct} and the fact that $ \bar G+\mu  \bar{X}\neq G^*+\mu X^*$, we see that $ (L-\mu)\langle (G^*+\mu X^*)-( \bar G+\mu  \bar{X}),  \bar{X}-X^*\rangle\geq\|( \bar G+\mu  \bar{X})-(G^*+\mu X^*)\|_F^2 
>0$. In particular, this implies $L > \mu$ and $\langle (G^*+\mu X^*)-( \bar G+\mu  \bar{X}),  \bar{X}-X^*\rangle > 0$, showing that $h$ in \cref{defnh1} is well defined. Furthermore, we have
\begin{equation}
        \label{boundqn}
        L-\mu\geq \frac{\|( \bar G+\mu  \bar{X})-(G^*+\mu X^*)\|_F^2}{\langle  (G^*+\mu X^*)-( \bar G+\mu  \bar{X}),  \bar{X}-X^*\rangle}. 
    \end{equation}

Now, it is routine to check that $h$ is $\mu$-strongly convex, $\nabla h( \bar{X})=- \bar G$, and $\nabla h(X^*)=-G^*$. Moreover, the relation in \eqref{eq:ce4} and hence the second-order stationarity of $(\bar U,\bar V)$ can be verified similarly to that in the case where $ \bar G+\mu  \bar{X}=G^*+\mu X^*$. Thus, it remains to show that $\nabla h$ is Lipschitz continuous with modulus $L$. 

To this end, notice that for all $X,Y\in \R^{m\times n}$ and the function $h$ defined in \cref{defnh1}, it holds that 
    \begin{align*}
        &\nabla^2h(X)[Y,Y]=L\sum_{i=1}^m\sum_{j\neq i}^n Y_{ij}^2+\mu \sum_{i=1}^{\revise{m}} Y_{ii}^2+\frac{(\langle Y, -G^*-\mu X^*+ \bar G+\mu  \bar{X} \rangle)^2}{\langle X^*- \bar{X}, -G^*-\mu X^*+ \bar G+\mu  \bar{X}\rangle} \\
        &\overset{\rm{(a)}}{=} L\sum_{i=1}^m\sum_{j\neq i}^n Y_{ij}^2+\mu \sum_{i=1}^{\revise{m}} Y_{ii}^2+\frac{(\langle \wDiag(Y_{11},\dots,Y_{mm}), -G^*-\mu X^*+ \bar G+\mu  \bar{X} \rangle)^2}{\langle X^*- \bar{X}, -G^*-\mu X^*+ \bar G+\mu  \bar{X}\rangle} \\
        &\overset{\rm (b)}{\leq} L\sum_{i=1}^m\sum_{j\neq i}^n Y_{ij}^2+\mu \sum_{i=1}^{\revise{m}} Y_{ii}^2+\frac{\| \wDiag(Y_{11},\dots,Y_{mm})\|^2_F\, \|-G^*-\mu X^*+ \bar G+\mu  \bar{X}\|_F^2}{\langle X^*- \bar{X}, -G^*-\mu X^*+ \bar G+\mu  \bar{X}\rangle} \overset{\rm (c)}{\leq} L\|Y\|_F^2,
    \end{align*}
    where in (a) we have used the fact that $ \bar{X},X^*, \bar G,G^*$ are diagonal (see \cref{eq:ce1}), in (b) we have used the Cauchy-Schwartz inequality, and in (c) we have used \eqref{boundqn}. This proves that $\nabla h$ is Lipschitz continuous with modulus $L$. 
    
    Finally, if $\|g\|_\infty > \lambda$, then by \cref{subd_nuclear} we know $\bar X$ is not a stationary point of $f$ in \cref{prob_nuclear} given such an $h$, and hence cannot be a global minimizer. Therefore, in this case, $f$ in \cref{prob_nuclear} is not $r$-factorizable.   
\end{proof}   

\begin{exmp}
\revise{We record a two-dimensional instance to illustrate the boundary mechanism in the preceding construction. Take $m=n=2$, $r=r^*=1$, $\lambda=\mu=1$, and $L=3$. We have $\kappa=L/\mu=3$. For $X=\begin{bmatrix}a&b\\ c&d\end{bmatrix}$, define}
\[
\revise{h(X):=\frac32(b^2+c^2)+\frac12(a-1)^2+\frac12d^2-a-4d+\frac12(a+d-1)^2.}
\]
\revise{The Hessian of $h$ is constant: on the $(a,d)$ variables it is $\begin{bmatrix}2&1\\1&2\end{bmatrix}$, whose eigenvalues are $1$ and $3$, and on the $(b,c)$ variables the curvature is $3$. Thus, $h$ is $1$-strongly convex and $\nabla h$ is $3$-Lipschitz continuous. Let}
\[
\revise{X^*=\begin{bmatrix}0&0\\0&2\end{bmatrix},\qquad \bar U=\bar V=e_1,\qquad \bar X:=\bar U\bar V^\top=\begin{bmatrix}1&0\\0&0\end{bmatrix}.}
\]
\revise{Since $I\in\partial\|X^*\|_*$,}
\[
\revise{\nabla h(X)=\begin{bmatrix}2a+d-3&3b\\3c&a+2d-5\end{bmatrix}\qquad\text{and} \qquad\nabla h(X^*)=-I,}
\]
\revise{we have $0\in \nabla h(X^*)+\partial\|X^*\|_*$. Hence $X^*$ is the unique global minimizer of $f(X)=h(X)+\|X\|_*$, and $\rank(X^*)=1$. On the other hand, $\nabla h(\bar X)=\Diag(\begin{bmatrix}
-1&-4
\end{bmatrix}^T)$, so $\nabla h(\bar X)\bar V+\bar U=0$ and $\nabla h(\bar X)^\top\bar U+\bar V=0$. For arbitrary $U,V,\xi,\eta\in\R^2$, by \cite[Equation (3.14)]{li2017geometry}, the Hessian of $F_1(U,V)=h(UV^\top)+\frac12(\|U\|^2+\|V\|^2)$ admits the chain-rule expression
\[
\revise{\begin{aligned}
&\nabla^2 F_1(U,V)[(\xi,\eta),(\xi,\eta)]\\
&=\nabla^2 h(UV^\top)[\xi V^\top+U\eta^\top,\xi V^\top+U\eta^\top]
  +2\langle \nabla h(UV^\top),\xi\eta^\top\rangle
  +\|\xi\|^2+\|\eta\|^2 .
\end{aligned}}
\]
\revise{For $U=(u_1,u_2)^\top$ and $V=(v_1,v_2)^\top$, in matrix form ordered by $(u_1,u_2,v_1,v_2)$, we have}
\[
\revise{\setlength{\arraycolsep}{3pt}
\nabla^2 F_1(U,V)=
\begin{bmatrix}
2v_1^2+3v_2^2+1 & v_1v_2 & 4u_1v_1+u_2v_2-3 & 6u_1v_2+u_2v_1\\
v_1v_2 & 3v_1^2+2v_2^2+1 & u_1v_2+6u_2v_1 & u_1v_1+4u_2v_2-5\\
4u_1v_1+u_2v_2-3 & u_1v_2+6u_2v_1 & 2u_1^2+3u_2^2+1 & u_1u_2\\
6u_1v_2+u_2v_1 & u_1v_1+4u_2v_2-5 & u_1u_2 & 3u_1^2+2u_2^2+1
\end{bmatrix}.}
\] Specializing this formula at $\bar U=\bar V=e_1$ with arbitrary perturbations $\xi=(\xi_1,\xi_2)^\top$ and $\eta=(\eta_1,\eta_2)^\top$ gives}  
\[
\revise{\nabla^2 F_1(\bar U,\bar V)[(\xi,\eta),(\xi,\eta)]
=3\xi_1^2+2\xi_1\eta_1+3\eta_1^2+4(\xi_2-\eta_2)^2\geq 0.}
\]
\revise{Thus $(\bar U,\bar V)$ is a second-order stationary point of $F_1$. However, $\bar X\neq X^*$ and hence is not a global minimizer of $f$. This gives a concrete boundary counterexample at $\kappa=3$.}
\end{exmp}

Combining \cref{propb1}, \cref{lemma5-2} and \cref{prop5-6}, we have the following result concerning the existence of the function $f$ in \cref{prob_nuclear} when $h\in \fS(L,\mu,r^*)$, which is not $r$-factorizable.
\begin{prop}
\label{nprop5-4}
    Let $L\geq \mu >0$, $r\in [m]$ and $r^*\in [m]\cup \{0\}$. Consider the following optimization problem:
    \begin{equation}
        \label{eq:opt_sc}
        \begin{aligned}
            &&\sup_{\begin{subarray}{c}
            x,g,y,v\in \R^{m}\\
            \tau\in \gp_m
            \end{subarray}} &\sum_{i=1}^m (L x_i+g_i)(\mu y_{\tau(i)}+v_{\tau(i)})+\sum_{i=1}^m (\mu x_i+g_i)(Ly_{\tau(i)}+v_{\tau(i)}) \\
            &&&-\sum_{i=1}^m(Lx_i+g_i)(\mu x_i+g_i)-\sum_{i=1}^m(Ly_i+v_i)(\mu y_i+v_i) \\
            &&s.t.\quad ~&\forall i\in [r],~x_i>0,~g_i=\lambda,\quad \forall i\in [r^*],~y_i>0,~v_i=\lambda,\\
            &&&\forall i\in [m]\setminus [r],~x_i=0,~g_i\in [0,\lambda+L\min_{j\in [r]}x_j],\\
            &&&\forall i\in [m]\setminus [r^*], ~y_i=0,~v_i\in [0,\lambda]. 
        \end{aligned}
\end{equation}
Then, the existence of $h\in \fS(L,\mu,r^*)$ such that the corresponding $f$ in \cref{prob_nuclear} is not $r$-factorizable, is equivalent to the existence of a feasible solution $(x,g,y,v,\tau)$ to \cref{eq:opt_sc} with $\|g\|_\infty>\lambda$ and a nonnegative objective function value.  
\end{prop}
\begin{proof}
    If there exists $h\in \fS(L,\mu,r^*)$ such that the corresponding $f$ in \cref{prob_nuclear} is not $r$-factorizable, then by \cref{propb1} we know there exists a feasible solution $(x,g,y,v,\tau)$ with $\|g\|_\infty>\lambda$ to \cref{eq:opt_sc} having nonnegative function value. Conversely, if there exists a feasible solution $(x,g,y,v,\tau)$ with $\|g\|_\infty>\lambda$ to \cref{eq:opt_sc} having nonnegative function value, by \cref{lemma5-2} we may assume that $x$ and $y$ are sorted in descending order \revise{and the subvectors $(g_i)_{i=r+1}^m$ and $(v_i)_{i=r^*+1}^m$ are also sorted in descending order}, and then by \cref{prop5-6} we know there exists $h\in \fS(L,\mu,r^*)$ such that $f$ in \cref{prob_nuclear} is not $r$-factorizable.
\end{proof}

Consequently, to prove \cref{thm5-9}, it suffices to solve the optimization problem \cref{eq:opt_sc} analytically. This is very technically and contains less insight, and hence we defer it to \cref{appa}. Finally, let us note that \cref{thm5-9} follows from \cref{nprop5-4}, \cref{propb2} and \cref{lemma5-7}.

\section{Generalization to the RIP case}
\label{sec:gen_rip}
When the condition in \cref{thm5-9} fails, we can find a function $h\in \fS(L,\mu,r^*)$ such that the corresponding $f$ in \cref{prob_nuclear} is not $r$-factorizable. Since $\fS(L,\mu,r^*)\subseteq \fS(L,\mu,r,r^*)$, this means the condition in \cref{thm5-9} is also necessary for the function class $ \fS(L,\mu,r,r^*)$, and it suffices to prove sufficiency in \cref{coro1-4}. Our proof is largely motivated by the following observation:
\begin{fact}
\label{fact6-1}
    If $h\in C^2(\R^{m\times n})$ and satisfies \cref{eq_RIP}, and the linear subspace $\cS\subseteq \R^{m\times n}$ contains only matrices of rank at most $q+r^*$, then the restriction $h|_{\cS}$ of $h$ on $\cS$ satisfies that $h|_{\cS}\in C^2(\cS)$ is $\mu$-strongly convex, and $\nabla h|_{\cS}$ is $L$-Lipschitz continuous. 
\end{fact}
Roughly speaking, to prove the sufficiency of the conditions in \cref{thm5-9} for the function class $\fS(L,\mu,r,r^*)$, we argue by contradiction. Suppose there exists $h\in \fS(L,\mu,r,r^*)$ such that the corresponding $f$ in \cref{prob_nuclear} is not $r$-factorizable, then there exist s a second-order stationary point $(\bar U,\bar V)\in \R^{m\times r}\times \R^{n\times r}$ such that $\bar X=\bar U\bar V^\top$ is not a global minimizer of $f$. If we can find a linear subspace $\cS$ such that $\cS$ contains $\bar X$ and $X^*$, and $\cS$ consists of matrices of rank no more than $r+r^*$, then we can apply \cref{thm5-9} to arrive at a contradiction by restricting $h$ and $f$ on $\cS$ and utilizing \cref{fact6-1}. This is achieved by the following lemma.
\begin{lemma}
 \label{common_svd}
     Let $A,B\in \R^{m\times n}$ with \revise{$\rank(A)+\rank(B)\leq k$}. Then, there exists $U\in \cO_m$ and $V\in \cO_n$ such that 
     \[   A=U \wDiag(\sigma(A))V^\top, \quad   B=U\begin{bmatrix}
         B_1 & 0 \\
         0 & 0
     \end{bmatrix} V^\top ,     \]
     where $B_1\in \R^{k\times k}$. 
 \end{lemma}
 \begin{proof}
     Suppose $\begin{bmatrix}
         A& B
     \end{bmatrix}=U_1D_1V_1^\top$ and $\begin{bmatrix}
         A \\
         B
     \end{bmatrix}=U_2D_2V_2^\top $ are the corresponding SVDs of $\begin{bmatrix}
         A & B
     \end{bmatrix}$ and $\begin{bmatrix}
         A \\
         B
     \end{bmatrix}$, respectively. Notice that both $\rank \begin{bmatrix}
         A &
         B
     \end{bmatrix}$ and $\rank\begin{bmatrix}
         A \\
         B
     \end{bmatrix}$ are no more than $k$. Then, the last $m-k$ rows of $U_1^\top A$ and $U_1^\top B$ are zero because $U_1^\top \begin{bmatrix}
         A& B
     \end{bmatrix}=D_1V_1^\top$, and the last $n-k$ columns of $AV_2$ and $BV_2$ are $0$ because $\begin{bmatrix}
         A \\
         B
     \end{bmatrix}V_2=U_2D_2$. This implies that 
     \begin{equation*}
         U_1^\top A V_2 = \begin{bmatrix}
             A_1 & 0 \\
              0 & 0 
         \end{bmatrix},\quad U_1^\top  BV_2=\begin{bmatrix}
             B_1 & 0 \\
             0 & 0
         \end{bmatrix}, \quad A_1,B_1\in \R^{k\times k}. 
     \end{equation*}
     Let an SVD of $A_1$ be given by $A_1=U_3 D_3 V_3^\top$. Now, it suffices to take 
     \[
     U=U_1\begin{bmatrix}
     U_3 &0\\0&I_{m-k}
     \end{bmatrix}\ \
     {\rm and}\ \
    V=V_2 \begin{bmatrix}
    V_3& 0\\0&I_{n-k}
    \end{bmatrix}. 
     \]
 \end{proof}
 We are now ready to prove in the next proposition the sufficiency of the conditions in \cref{thm5-9} for the function class $\fS(L,\mu,r,r^*)$. With \cref{common_svd} in hand, the remaining task is to ensure that we are also able to restrict $F_r$ in \cref{prob_relaxation} on a corresponding linear subspace to preserve the second-order stationarity of $(\bar U,\bar V)$, which is the key ingredient in the proof.
\begin{prop}
\label{prop_rip}
    Assume $h\in \fS(L,\mu,r,r^*)$ and let $X^*\in \R^{m\times n}$ with $\rank(X^*)=r^*$ being a global minimizer of $f$ in \cref{prob_nuclear}. Set $\kappa=\frac{L}{\mu}\geq 1$. Assume that $F_r$ in \cref{prob_relaxation} has a second-order stationary point $(\bar U,\bar V)$ with $\bar X:=\bar U\bar V^\top \neq X^*$, then both of the following conditions fail:
    \begin{itemize}
        \item[(1)] $r=m$;
        \item[(2)] $r\geq r^*$ and $\min\{r,m-r^*\}>\frac{(\kappa-1)^2}{4}\min\{r^*,m-r\}.$
    \end{itemize}
\end{prop}
\begin{proof}
    Let $k=r+r^*$.  If $k\geq m$, then the result follows immediately from \cref{thm5-9}. Therefore, we assume $k<m$ in the following. Then $r<m$, which proves that $r=m$ is false. Since $k\geq \rank(\bar X)+\rank(X^*)$, we can apply \cref{common_svd} to show that there exist $R\in \cO_m$ and $Q\in \cO_n$ such that 
    \begin{align*}
        R^\top X^*Q=\begin{bmatrix}
        X^*_1 & 0 \\
          0 & 0
    \end{bmatrix},\quad  R^\top\bar XQ=\begin{bmatrix}
        \bar X_1 & 0 \\
         0 & 0
    \end{bmatrix},
    \end{align*}
     where $X_1^*, \bar X_1\in \R^{k\times k}$. Then $\begin{bmatrix}
        X_1^* & 0 \\
         0 & 0
    \end{bmatrix}$ is a stationary point of $\tilde f$, where $\tilde f(W):=f(RW Q^\top)=\tilde h(W)+\lambda \|W\|_*$ with $\tilde h(W) := h(RWQ^\top)$. Clearly, we have $\tilde h\in \fS(L,\mu,r,r^*)$ and $\tilde F_r(U,V):=\tilde h(UV^\top)+\frac{\lambda}{2}\left( \|U\|_F^2+\|V\|_F^2\right) = F_r(RU, QV)$ for all $U\in \R^{m\times r}$ and $V\in \R^{n\times r}$. In particular, $(R^\top \bar U, Q^\top\bar V)$ is a second-order stationary point of $\tilde F_r$ with $R^\top \bar U \bar V^\top Q = R^\top \bar X Q\neq  \begin{bmatrix}
        X_1^* & 0 \\
         0 & 0
    \end{bmatrix}$. Consequently, replacing $f$ by $\tilde f$ if necessary, we may assume that 
    \begin{equation}
    \label{blockxbarstar}
        X^*=\begin{bmatrix}
        X^*_1 & 0 \\
         0 & 0
    \end{bmatrix},\quad \bar X=\begin{bmatrix}
        \bar X_1 & 0 \\
         0 &  0
    \end{bmatrix}.
    \end{equation}
     Let us also note that since $\dim \cR(\bar U)=\rank(\bar U)=\rank (\bar V)=\rank(\bar X)=\dim \cR(\bar X)$ by \cref{haha}(ii), and $\cR(\bar X)\subseteq \cR(\bar U)$, we know that $\cR(\bar U)=\cR
    (\bar X)$. This implies that $\bar U=\begin{bmatrix}
        \bar U_1 \\
          0
    \end{bmatrix}$ with $\bar U_1\in \R^{k\times r}$. Similar arguments on $\bar X^\top$ show that $\bar V=\begin{bmatrix}
        \bar V_1 \\
        0 
    \end{bmatrix}$ with $\bar V_1\in \R^{k\times r}$. 
    
    Define a function $\widehat f:\R^{k\times k}\to \R$ and $\widehat h:\R^{k\times k}\to \R$ as:
    \begin{align*}
      \widehat h(K)&:=h\left(\begin{bmatrix}
            K & 0 \\
             0 & 0
        \end{bmatrix}\right), \\
        \widehat f(K)&:=f\left(\begin{bmatrix}
            K & 0 \\
             0 & 0
        \end{bmatrix}\right)= \widehat h(K) + \lambda\left\|\begin{bmatrix}
            K & 0 \\
             0 & 0
        \end{bmatrix}  \right\|_* =\widehat h(K)+\lambda \|K\|_*
    \end{align*}
   The function $\widehat f$ can be viewed as the restriction of $f$ to the linear subspace $\cS:=\begin{bmatrix}
        \R^{k\times k} & 0  \\
         0 & 0
    \end{bmatrix}$. 
    Using \cref{fact6-1}, we know that $\widehat h$ is $\mu$-strongly convex, and $\nabla \widehat h$ is Lipschitz continuous with modulus $L$. Next, observe that $X_1^*$ is a global minimizer of $\widehat f$ and hence we have $0\in \partial\widehat f(X_1^*)$. Since $\widehat f(\cdot)=\widehat h(\cdot)+\lambda\|\cdot\|_*$, we see that $\widehat f$ is $\mu$-strongly convex, and hence $X_1^*$ is the unique stationary point of $\widehat f$. Moreover, using \cref{blockxbarstar}, we know that $\rank(X_1^*)=\rank(X^*)=r^*$. Consequently, $\widehat h\in C^2(\R^{k\times k})$ satisfies that $\widehat h\in \fS(L,\mu,r^*)$. Next, by direct calculation, the corresponding $\widehat F_r$ in \cref{prob_relaxation} can be written as:
    \begin{equation}
        \widehat F_r(U,V):=\widehat h(UV^\top)+\frac{\lambda}{2}\left( \|U\|_F^2+\|V\|_F^2\right)=h\left(\begin{bmatrix}
            UV^\top & 0 \\
             0 & 0
        \end{bmatrix}\right)+\frac{\lambda}{2}\left( \|U\|_F^2+\|V\|_F^2\right),
    \end{equation}
    which can be viewed as the restriction of $F_r$ on the linear subspace $\cS_1:= \begin{bmatrix}
        \R^{k\times r} \\
        0
    \end{bmatrix} \times \begin{bmatrix}
        \R^{k\times r} \\
        0
    \end{bmatrix}$. Since $(\bar U,\bar V)\in \cS_1$ is a second-order stationary point of $F_r$, we immediately see that $(\bar U_1,\bar V_1)$ is a second-order stationary point of $\widehat F_r$, since restricting a function on a linear subspace would not change second-order stationarity. The assumption $\bar X\neq X^*$ gives that $\bar X_1\neq X^*_1$ given the representation in \cref{blockxbarstar}, and hence $\bar X_1$ is not a global minimizer of $\widehat f$, since $\widehat f$ is strongly convex and its unique global minimizer is $X^*_1$. Consequently, $\widehat f$ is not $r$-factorizable. \revise{If $r^*=0$, then $k=r$, and hence the first sufficient condition in \cref{thm5-9}, applied to the restricted $k\times k$ problem, holds. This would imply that $\widehat f$ is $r$-factorizable, contradicting the preceding conclusion. Therefore $r^*>0$.} \revise{By \cref{thm5-9}, applied to the restricted problem, the condition $r\geq r^*$ and $r> \frac{(\kappa-1)^2}{4}r^*$ must fail.} \revise{Since $r+r^*=k<m$, we know that $\min\{r,m-r^*\}=r$ and $\min\{r^*,m-r\}=r^*$. Hence this failed restricted condition is equivalent to condition (2) in the present proposition. Since $r<m$ was already shown above, condition (1) also fails, and the proof has been completed.}
\end{proof}